\documentclass{article}%
\usepackage{amsmath}
\usepackage{amsfonts}
\usepackage{amsthm}
\usepackage{lineno,hyperref}
\usepackage{amssymb}
\usepackage{graphicx}
\usepackage{epstopdf}
\usepackage{caption}
\usepackage{subcaption}%
\usepackage{natbib}
\usepackage{authblk}

\newtheorem{hypothesis}{Hypothesis}
\newtheorem{theorem}{Theorem}
\newtheorem{proposition}[theorem]{Proposition}
\newtheorem{lemma}[theorem]{Lemma}

\newtheorem{definition}{Definition}
\newtheorem{remark}{Remark}
\begin{document}

{\centering%
\noindent\begin{tabular}{|l|}
\hline
L. Sanz, R. Bravo de la Parra, M. Marv\'a, E. S\'anchez. Non-linear population\\
discrete models with two time scales: re-scaling of part of the slow process.\\
Advances in Difference Equations, 2019(1), 401., 2019.\\
http://doi.org/10.1186/s13662-019-2303-1\\
\hline
\end{tabular}
}%

\title{Non-linear population discrete models with two time scales: re-scaling of part of the slow process}
\author[1]{Luis Sanz\thanks{luis.sanz@upm.es}}
\author[2]{Rafael Bravo de la Parra\thanks{rafael.bravo@uah.es}}
\author[3]{Marcos Marv\'a\thanks{marcos.marva@uah.es}}
\author[4]{Eva S\'anchez\thanks{evamaria.sanchez@upm.es}}

\affil[1]{Depto. Matem\'aticas, E.T.S.I Industriales, Technical University of Madrid, Madrid, Spain.}
\affil[2]{U.D. Matem\'aticas, Universidad de Alcalá, Alcal\'a, Spain}
\affil[3]{U.D. Matem\'aticas, Universidad de Alcalá, Alcal\'a, Spain}
\affil[4]{Depto. Matem\'aticas, E.T.S.I Industriales, Technical University of Madrid, Madrid, Spain.}
\date{}

\maketitle

\begin{abstract}
In this work we present a reduction result for discrete time systems with two time scales. In order to be valid, previous results in the field require some strong hypotheses that are difficult to check in practical applications. Roughly speaking, the iterates of a map as well as their differentials must converge uniformly on compact sets. Here, we eliminate the hypothesis of uniform convergence of the differentials at no significant cost in the conclusions of the result.
This new result is then used to extend to nonlinear cases the reduction of some population discrete models involving processes acting at different time scales. In practical cases, some processes that occur at a fast time scale are often only measured at slow time intervals, notably mortality. For a general class of linear models that include such kind of processes, it has been shown that a more realistic approach requires the re-scaling of those processes to be considered at the fast time scale. We develop the same type of re-scaling in some nonlinear models and prove the corresponding reduction results. We also provide an application to a particular model of a structured population in a two-patch environment.
\end{abstract}
%
{\bf keywords:}discrete-time system, time scales, re-scaling, structured population.

\section{Introduction\label{sec1}}

We consider discrete systems in the framework of population dynamics models.
The complexity of these models can be treated distinguishing the different
time scales at which the different processes involved act. In an idealization
of this approach we proposed (see \cite{Bravo13} and references therein) to
merge two different processes acting at different time scales in a single
model as we describe next. The effect of the fast process during a fast time
unit is represented by a general map $F$, and a general map $S$ describes the
slow process in a slow time unit. We choose this latter as the time unit of
the common discrete model. If the slow time unit is approximately $k$ times
larger than the fast one, we consider that during this interval the fast
process acts sequentially $k$ times followed by the slow process acting once.
The combined effect of both processes during a slow time unit is therefore
represented by the composition of map $S$ and the k-th iterate $F^{(k)}$ of
map $F$. If we let vector $X(t)$ represent the population state at time $t$,
the general form of the system is:
\begin{equation}
X(t+1)=S\left(  F^{(k)}\left(  X(t)\rule{0ex}{2ex}\right)  \right)
\label{1-sg}%
\end{equation}
The subsequent issue is to take advantage of the existence of time scales to
reduce the proposed system. The reduction we are referring to can be
considered to be part of the so-called methods of aggregation of variables,
that exist in different mathematical settings \cite{Auger08a,Auger08b}. It
consists on finding a certain number of global variables, which are functions
of the state variables, and a reduced system, approximately describing their
dynamics, such that it is possible to get information on the asymptotic
behaviour of the solutions of the original system in terms of this reduced
system. This procedure has a direct interpretation in terms of ecological
hierarchy theory \cite{Auger03} and the concept of up-scaling through
ecological hierarchical levels \cite{Lischke07,Levin92}. In general,
aggregation methods are approximate in the sense that quantitatively the
dynamics of the original system can not be described exactly by the dynamics
of the reduced one. However, the error incurred decreases when the ratio of
time scales increases and, for a large enough value of this ratio, the two
systems have the same qualitative dynamics.

In \cite{Sanz08} it is proved that system \eqref{1-sg} can be reduced if the
limit of the iterates of map $F$ exists and can be expressed as the
composition of two maps going through a lower dimensional space. With
additional hypotheses it is possible to extract information on the asymptotic
behaviour of the solutions of system \eqref{1-sg} by means of the reduced
system. To be specific, one can study the existence, stability and basins of
attraction of steady states and periodic solutions of the original system by
performing the study for the corresponding aggregated system. The additional
hypotheses essentially consist in the convergence on compact sets of the $F$
iterates and their differentials. However, this latter assumption is generally
too involved, if not impossible, to be proved in particular applications.

The main result in this work proves that the mere convergence on compact sets
of the $F$ iterates is enough to obtain almost the same results as in
\cite{Sanz08}. The downside is that we cannot guarantee the convergence of the
system dynamics to the attractor, equilibrium or periodic solution, obtained
through the analysis of the reduced system. However, we prove that this
dynamics remains as close to the attractor as desired if the ratio of time
scales is sufficiently large. Therefore, from the point of view of population
dynamics models this last property is good enough to obtain valuable
qualitative results.

This reduction result is applied in the context of discrete-time models of
structured metapopulations with two time scales. The aim is to extend to the
nonlinear case certain results obtained in \cite{Nguyen11} for linear models.
In structured metapopulation models that distinguish time scales, biological
evidence suggest associating the slow one to the local demography (maturation,
survival, reproduction) and the fast one to the movements between patches.
When we express this in the form of system \eqref{1-sg} we are representing
that individuals perform at first a series of $k$ dispersal events between
patches followed by a demographic event in the arrival patch. This way of
separating slow and fast processes seems acceptable as far as reproduction is
concerned, particularly in the case of populations with
non-overlapping generations. Nevertheless, it is arguable whether survival
should be considered at the slow time scale because deaths may occur at any
moment of a slow time interval, i.e., in any of the patches through which
individuals pass during this interval. In order to include this issue in the
model, and having in mind that survival data may only exist for slow time
units, in \cite{Nguyen11} it is proposed to move survival from the slow to the
fast process by approximating its effect during the fast time unit, i.e.,
\textit{re-scaling} survival to the fast time scale.

According to this discussion, we present two general discrete-time models of
structured metapopulations with two time scales. In the first one we consider
the effect of survival in the time step corresponding to the slow process
and, in the second one, in that corresponding to the fast process. The
main result of this work makes it possible to extend the reduction method to
the latter case.

This general framework is used to illustrate the effects of fast dispersals on
local demography that emerge at the global metapopulation level. To this end,
we extend the non-spatial model presented in \cite{Veprauskas17} to a habitat
with two patches between which fast dispersals are considered. The model is
structured into three classes: juveniles and two adult stages. Adults that
reproduce at the end of an interval of time do not reproduce at the end of the
following one. A fraction of those not having reproduced do so at the end of
the following interval. Thus, adults are classified into those reproducing at
the end of the time interval, active adults, and those who do not, inactive
adults. The aim of this consideration is studying the reproductive synchrony
of adults. The model in \cite{Veprauskas17} is a discrete time nonlinear
matrix model whose inherent projection matrix is imprimitive \cite{Cushing15}.
This entails that the stability of the extinction equilibrium when $R_{0}$
increases through 1 bifurcates either to the stability of a non-extinction
equilibrium or of a synchronous 2-cycle. The first option represents adults
reproductive asynchrony whereas the second one is associated to reproductive synchrony.

Both metapopulation models, with or without re-scaled survival, and with local
dynamics based on the model \cite{Veprauskas17}, can be studied through their
associated reduced systems. Indeed, these have the same functional form as the
system in \cite{Veprauskas17} so that the results therein can be applied.

The main result on systems reduction is developed in Section \ref{sec2}. Two
general discrete-time models of structured metapopulations with two time
scales are presented in Section \ref{sec3}. The first one considers the
dispersal process to be fast and the local demographic one, including
survival, to be slow, whilst in the second one survival is transferred from
the slow to the fast process. These two general models are applied in Section
\ref{sec4} to the particular case of a population structured into three
classes inhabiting in a two-patch environment. With the help of these two
particular models we illustrate how the effect of fast dispersals can make
emerge at the global level asymptotic properties different from those
occurring at the local level. The models are also used to show that the
modelling choice regarding survival can imply drastic changes in behavior. A
discussion and an appendix with the proofs of some results complete the work.

\section{Reduction of slow-fast discrete systems\label{sec2}}

The point of departure is the paper \cite{Sanz08} so we recall the
presentation therein.

Let $N\in\mathbb{Z}_{+}\ $and let $\Omega_{N}\subset\mathbb{R}^{N}$ be a set
with non-empty interior. We start by introducing the \textit{original} or
\textit{complete} model in the following general form:
\begin{equation}
X_{k}(t+1)=H_{k}\left(  X_{k}(t)\right)  \label{modorig}%
\end{equation}
where $k\in\mathbb{Z}_{+}$ and $H_{k}:\Omega_{N}\longrightarrow\Omega
_{N},\ H_{k}\in\mathcal{C}^{1}(\Omega_{N})$. System \eqref{1-sg} is a
particular case of system \eqref{modorig} where $H_{k}=S\circ F^{(k)}$.

In order to carry out the reduction of the model, we assume the following conditions:

\begin{hypothesis}
\label{H1} The following pointwise limit exists in $\Omega_{N}$
\begin{equation}
\lim_{k\rightarrow\infty}H_{k}(X)=:H(X)\in\mathcal{C}^{1}(\Omega_{N}).
\label{pwl}%
\end{equation}

\end{hypothesis}

\begin{hypothesis}
\label{H2} There exist a subset $\Omega_{q}\subset\mathbb{R}^{q}$ with $q<N$
and two maps \newline$G:\Omega_{N}\rightarrow\Omega_{q}$ and $T:\Omega
_{q}\rightarrow\Omega_{N}$ of class $\mathcal{C}^{1}$ on their respective
domains such that
\begin{equation}
\label{decomp}H=T\circ G.
\end{equation}

\end{hypothesis}

The approximate reduction of system (\ref{modorig}) is carried out in two
steps. First, we define the auxiliary system
\begin{equation}
X(t+1)=H(X(t)). \label{modaux}%
\end{equation}
Applying $G$ to both members of the previous expression we have
\[
G(X(t+1))=G\circ H(X(t))=G\circ T\circ G(X(t)),
\]
and so by defining the global variables%
\begin{equation}
Y(t):=G(X(t))\in\mathbb{R}^{q}, \label{gobvar}%
\end{equation}
we obtain the reduced or aggregated system
\begin{equation}
Y(t+1)=\bar{H}(Y(t)), \label{modagreg}%
\end{equation}
where we have introduced the notation $\bar{H}:=G\circ T$.

Note that through this procedure we have constructed an approximation that
allows us to reduce a system with $N$ variables to a new system with $q$
variables. In most practical applications, $q$ is much smaller than $N$.

The work \cite{Sanz08} presents results that allow one to relate the existence
of equilibrium points and periodic orbits for systems (\ref{modorig}) and
(\ref{modagreg}). More specifically, if certain conditions are met (see below)
and the aggregated system has a hyperbolic $p$-periodic point $Y^{\ast}$
($p\in\mathbb{N}$), then for a large enough value of $k$ the original system
has a hyperbolic associated $p$-periodic point $X_{k}^{\ast}$ that can be
approximated in terms of $Y^{\ast}$. Moreover, $X_{k}^{\ast}$ is (locally)
asymptotically stable (resp. unstable) if and only if $Y^{\ast}$ is (locally)
asymptotically stable (resp. unstable) and, in the first case, the basin of
attraction of $X_{k}^{\ast}$ can be approximated in terms of that
corresponding to $Y^{\ast}$.

For the previous results to hold, in addition to Hypotheses \ref{H1} and
\ref{H2} one must include two additional conditions. The first one is:

\begin{hypothesis}
\label{H3} The limit in \eqref{pwl} is uniform on all compact sets of
$\Omega_{N}$.
\end{hypothesis}

The second condition is that $\lim_{k\rightarrow\infty}DH_{k}=DH$ uniformly on
compact sets of $\Omega_{N}$, where $DH$ denotes the differential.

In practice, when applying approximate reduction techniques to population
models even in simple settings, the (uniform) convergence of the differentials
required in the second condition is difficult to check and involves lengthy
reasonings and calculations (see for example the proofs in
\cite{Marva13,Marva09}). In more general situations, proving that the
condition holds is completely unfeasible. This shows the need for results that
allow one to drop this condition and still offer a relationship between the
aggregated and the original model.

Next we present a result of this type where, in the case in which the reduced
system (\ref{modagreg}) has $m$-periodic hyperbolic points, the local dynamics
of the original system (\ref{modorig}) can be characterized in terms of them.
Roughly speaking, and restricting our attention to the case of equilibrium
points, it states that if the aggregated system (\ref{modagreg}) has a
hyperbolic locally asymptotically stable equilibrium point $Y^{\ast}$, then if
we choose any sufficiently small neighbourhood $U$ of the point $X^{\ast
}:=T(Y^{\ast})$ and the value of $k$ is large enough, there is an equilibrium
point $X_{k}^{\ast}$ of system (\ref{modorig}) in $U$ and all trajectories
starting in $U$ do not leave it. Also, if $Y^{\ast}$ is hyperbolic and
unstable for (\ref{modagreg}) then any neighbourhood of $X^{\ast}$ is unstable
\cite{Lasalle76} provided the value of $k$ is large enough.

Let us first recall the following definition:

\begin{definition}
Let $\Omega\subset\mathbb{R}^{N}$ and let $f:\Omega\rightarrow\Omega$. We say
that a compact set $C\subset\Omega$ is a trapping region for system
$X(t+1)=f(X(t))$ whenever $f(C)\subset\ $int$(C)$ (interior of $C$).
\end{definition}

For each $\delta>0$ and $X\in\mathbb{R}^{N}$ let us denote $B(X,\delta
):=\left\{  x\in\mathbb{R}^{N}:\left\Vert x-X\right\Vert <\delta\right\}  $
where $\left\Vert \ast\right\Vert $ denotes the Euclidean norm in
$\mathbb{R}^{N}$ and let $\bar{B}(X,\delta)$ denote the adherence of
$B(X,\delta)$. Now we have:

\begin{theorem}
\label{th1} Let us assume Hypotheses \ref{H1}, \ref{H2} and \ref{H3}. Let
$m\in\mathbb{N\ }$, let $Y^{\ast}\in\mathbb{R}^{q}$ be an hyperbolic
$m$-periodic point of system (\ref{modagreg}) and let $X^{\ast}:=T(Y^{\ast})$.

\begin{enumerate}
\item Assume $Y^{\ast}$ is locally asymptotically stable and let $X_{0}%
\in\Omega_{N}$ be such that $Y_{0}:=G(X_{0})$ satisfies $\lim_{n\rightarrow
\infty}\bar{H}^{mn}(Y_{0})=Y^{\ast}$. Then it follows:

(1.a.) There exists $\delta_{0}>0$ such that for every $\delta>0$ satisfying
$0<\delta<\delta_{0},$ there exists $k_{\delta}\in\mathbb{Z}_{+}$ such that
for all $k\geq k_{\delta}$, $\bar{B}(X^{\ast},\delta)$ is a trapping region
for mapping $H_{k}^{m}$ and, moreover, $H_{k}^{m}$ has at least a fixed point
$X_{k}^{\ast}$ in $B(X^{\ast},\delta)$.

(1.b.) There exists $\delta_{0}>0$ such that, for every $\delta>0$ satisfying
$0<\delta<\delta_{0},$ there exist $n_{\delta},k_{\delta}\in\mathbb{Z}_{+}$
such that $H_{k}^{mn+1}(X_{0})\in B(X^{\ast},\delta)$ for all $k\geq
k_{\delta}$ and $n\geq n_{\delta}$.

\item Assume $Y^{\ast}$ is unstable. Then there exists $\delta_{0}>0$ such
that for every $\delta>0$ satisfying $0<\delta<\delta_{0}$ and any point $X$
such that $\left\Vert X-X^{\ast}\right\Vert =\delta,$ there exists $k_{\delta
}\in\mathbb{Z}_{+}$ such that $H_{k}^{m}(X)\notin\bar{B}(X^{\ast},\delta)$ for
all $k\geq k_{\delta}$. In particular $\bar{B}(X^{\ast},\delta)$ is not a
stable set for $H_{k}^{m}(X)$.
\end{enumerate}
\end{theorem}

\begin{proof}
In \cite{Sanz08} it is proved, using only Hypotheses \ref{H1} and \ref{H2},
that if $Y^{\ast}$ is a $m$-periodic point of system (\ref{modagreg}) then
$X^{\ast}:=T(Y^{\ast})$ is a $m$-periodic point of system (\ref{modaux}) and,
moreover, $\rho(D\bar{H}^{m}(Y^{\ast}))=\rho(DH^{m}(X^{\ast})),$ where
$\rho(A)$ denotes the spectral radius of matrix $A$, so that if $Y^{\ast}$ is
hyperbolic asymptotically stable (resp. hyperbolic unstable) for
(\ref{modagreg}) then $X^{\ast}$ is hyperbolic a.s. (resp. hyperbolic u.) for
(\ref{modaux}).

\noindent\textbf{1.a.} Let us assume that $Y^{\ast}$ is a hyperbolic a.s.
$m$-periodic point for system (\ref{modagreg}) and therefore so is $X^{\ast}$
for system (\ref{modaux}). Let $0<\gamma<1$ be such that $\rho(DH^{m}(X^{\ast
}))<\gamma$. Then, it is well known that there exist $\delta_{0}>0$ and a
consistent matrix norm $\Vert\cdot\Vert$ in $\mathbb{R}^{N\times N}$ such that
$\left\Vert DH^{m}(X)\right\Vert \leq\gamma$ for every $X\in\bar{B}(X^{\ast
},\delta_{0})$. Let $\delta$ be such that $0<\delta<\delta_{0}$. Then, if
$X\in\bar{B}(X^{\ast},\delta)$
\[
\Vert H^{m}(X)-X^{\ast}\Vert=\Vert H^{m}(X)-H^{m}(X^{\ast})\Vert\leq
\gamma\Vert X-X^{\ast}\Vert\leq\gamma\delta.
\]
From the uniform convergence of $H_{k}$ to $H$ on compact sets it follows (see
Lemma \ref{lema1} in the Appendix) that there exists $k_{\delta}\in
\mathbb{Z}_{+}$ such that for all $k\geq k_{\delta}$,
\[
\sup_{X\in\bar{B}(X^{\ast},\delta)}\left\Vert H_{k}^{m}(X)-H^{m}(X)\right\Vert
<\delta(1-\gamma).
\]
Then for every $X\in\bar{B}(X^{\ast},\delta)$ and every $k\geq k_{\delta} $
\[%
\begin{array}
[c]{l}%
\left\Vert H_{k}^{m}(X)-X^{\ast}\right\Vert \leq\left\Vert H_{k}^{m}%
(X)-H^{m}(X)\right\Vert +\left\Vert H^{m}(X)-X^{\ast}\right\Vert \leq\\
\rule{4ex}{0ex}\leq\sup_{X\in\bar{B}(X^{\ast},\delta)}\left\Vert H_{k}%
^{m}(X)-H^{m}(X)\right\Vert +\left\Vert H^{m}(X)-X^{\ast}\right\Vert
<\delta(1-\gamma)+\gamma\delta=\delta,
\end{array}
\]
so that $\bar{B}(X^{\ast},\delta)$ is a trapping region for mapping $H_{k}%
^{m}$ as we wanted to prove.

Moreover, since for $k\geq k_{\delta}$ the set $\bar{B}(X^{\ast},\delta)$ is
convex, compact and positively invariant for $H_{k}^{m},$ then Brower's fixed
point theorem \cite{border1989fixed} assures that there exists at least a
fixed point $X_{k}^{\ast}$ for $H_{k}^{m}$ in $\bar{B}(X^{\ast},\delta)$.

\vspace{3ex}

\noindent\textbf{1.b.} Let $\delta_{0}$ be like in (1.a) and let $\delta$ be
such that $0<\delta<\delta_{0}$. Now using (1.a) we know that there exists
$k^{\prime}_{\delta}\in\mathbb{Z}_{+}$ such that for all $k\geq k^{\prime
}_{\delta}$,
\begin{equation}
\label{211}H_{k}^{m}(\bar{B}(X^{\ast},\delta))\subset B(X^{\ast},\delta).
\end{equation}

Let $X_{0}\in\Omega_{N}$ be such that $Y_{0}:=G(X_{0})$ satisfies
$\lim_{n\rightarrow\infty}\bar{H}^{mn}(Y_{0})=Y^{\ast}$. Thus, the continuity
of $T$ implies $\lim_{n\rightarrow\infty}T(\bar{H}^{mn}(Y_{0}))=T(Y^{\ast})$.
Using that $T\circ\bar{H}^{n-1}\circ G=H^{n}$ (Proposition 3.3. in
\cite{Sanz08}) and the fact that $G(X_{0})=Y_{0}$ it follows
\[
\lim_{n\rightarrow\infty}H^{mn+1}(X_{0})=X^{\ast}.
\]
Then, there exists $n_{\delta}\in\mathbb{Z}_{+}$ such that for all $n\geq
n_{\delta}$,
\begin{equation}
\left\Vert H^{mn+1}(X_{0})-X^{\ast}\right\Vert <\delta/2. \label{213}%
\end{equation}

The uniform convergence of $H_{k}$ to $H$ on compact sets of $\Omega_{N}$
implies, using Lemma \ref{lema1}, that $\lim_{k\rightarrow\infty}%
H_{k}^{mn_{\delta}+1}(X_{0})=H^{mn_{\delta}+1}(X_{0})$. Therefore, there
exists $k^{\prime\prime}_{\delta}\in\mathbb{Z}_{+}$ such that for all $k\geq
k^{\prime\prime}_{\delta}$
\begin{equation}
\label{214}\left\Vert H_{k}^{mn_{\delta}+1}(X_{0})-H^{mn_{\delta}+1}%
(X_{0})\right\Vert <\delta/2
\end{equation}

Using now (\ref{213}) and (\ref{214}) we have that for all $k\geq
k^{\prime\prime}_{\delta}$
\begin{align*}
\left\Vert H_{k}^{mn_{\delta}+1}(X_{0})-X^{\ast}\right\Vert  &  \leq\left\Vert
H_{k}^{mn_{\delta}+1}(X_{0})-H^{mn_{\delta}+1}(X_{0})\right\Vert +\\
&  \rule{2ex}{0ex}+\left\Vert H_{k}^{mn_{\delta}+1}(X_{0})-H^{mn_{\delta}%
+1}(X_{0})\right\Vert <\delta/2+\delta/2 =\delta,
\end{align*}
that is, $H_{k}^{mn_{\delta}+1}(X_{0})\in B(X^{\ast},\delta)$. Finally, taking
$k_{\delta}=\max(k^{\prime}_{\delta},k^{\prime\prime}_{\delta})$ and using
(\ref{211}) we have that for all $k\geq k_{\delta}$, and all $n\geq n_{\delta
}$, $H_{k}^{mn+1}(X_{0})\in B(X^{\ast},\delta)$ and so the result is proved.

\vspace{3ex}

\noindent\textbf{2. }Let $\left\Vert \ast\right\Vert $ be any norm in
$\mathbb{R}^{N}$. Let $\lambda$ be an eigenvalue $\lambda$ of $DH^{m}(X^{\ast
})$ such that $\left\vert \lambda\right\vert >1$ and let $u$ be an unit vector
belonging to the corresponding eigenspace. Now let $\mu>0$ be such that
$\gamma:=\left\vert \lambda\right\vert -\mu$ is larger than 1.

We know that
\[
H^{m}\left(  X\right)  =X^{\ast}+DH^{m}(X^{\ast})\left(  X-X^{\ast}\right)
+\left\Vert X-X^{\ast}\right\Vert G(X),
\]
where $G$ is continuous in a neighbourhood of $X^{\ast}$, so that there exists
$\delta_{0}>0$ such that for all $X$ satisfying $\left\Vert X-X^{\ast
}\right\Vert <\delta_{0}$ we have $\left\Vert G(X)\right\Vert <\mu$.

Now let $\delta$ be such that $0<\delta<\delta_{0}$ and let $X:=X^{\ast
}+\delta u\in\bar{B}(X^{\ast},\delta)$. Then
\begin{align*}
\rule{-5ex}{0ex}\left\Vert H^{m}(X)-X^{\ast}\right\Vert  &  =\left\Vert
H^{m}(X^{\ast}+\delta u)-X^{\ast}\right\Vert =\left\Vert \delta DH^{m}%
(X^{\ast})u+\left\Vert \delta u\right\Vert G(X^{\ast}+\delta u)\right\Vert
\geq\\
&  \geq\delta\left\Vert DH^{m}(X^{\ast})u\right\Vert -\left\Vert \delta
u\right\Vert \left\Vert G(X^{\ast}+\delta u)\right\Vert =\\
&  =\delta\left\vert \lambda\right\vert \left\Vert u\right\Vert -\delta
\left\Vert G(X^{\ast}+\delta u)\right\Vert >\delta\left\vert \lambda
\right\vert -\delta\mu=\gamma\delta.
\end{align*}
From the uniform convergence of $H_{k}^{m}$ to $H^{m}$ on compact sets (Lemma
\ref{lema1} in Appendix) it follows that there exists $k_{\delta}\in
\mathbb{Z}_{+}$ such that for all $k\geq k_{\delta}$,%
\[
\sup_{X\in\bar{B}(X^{\ast},\delta)}\left\Vert H_{k}^{m}(X)-H^{m}(X)\right\Vert
<\delta(\gamma-1).
\]
Then if $X$ is such that $\left\Vert X-X^{\ast}\right\Vert =\delta$ we have
that for all $k\geq k_{\delta}$
\[%
\begin{array}
[c]{l}%
\left\Vert H_{k}^{m}(X)-X^{\ast}\right\Vert \geq\left\Vert H^{m}(X)-X^{\ast
}\right\Vert -\left\Vert H_{k}^{m}(X)-H^{m}(X)\right\Vert \geq\\
\rule{4ex}{0ex}\geq\left\Vert H^{m}(X)-X^{\ast}\right\Vert -\sup_{X\in\bar
{B}(X^{\ast},\delta)}\left\Vert H_{k}^{m}(X)-H^{m}(X)\right\Vert >\gamma
\delta-\delta(\gamma-1)=\delta,
\end{array}
\]
as we wanted to prove.
\end{proof}

\begin{remark}
For fixed $k\geq k_{\delta}$ it is not possible to claim that $X_{k}^{\ast}$
is attracting or stable. However, Theorem \ref{th1} provides information about
the original system which is, for all practical purposes, as valuable as that
provided by the results of \cite{Sanz08}, and has the advantage of dropping
the above mentioned stringent condition on the convergence of the differentials.
\end{remark}

\section{Two time scales structured metapopulations models: re-scaling
survival to the fast time unit\label{sec3}}

In this section we present two density dependent discrete population models
whose dynamics is driven by two processes, slow and fast. The population is
considered structured into $q$ stages and inhabiting an environment divided
into $r$ patches. In the first general model that we propose, the fast process
includes the movements of the individuals between patches and the slow process
consists in all demographic issues. After this, we propose a new model where
we carry out the re-scaling of the death process to the fast time unit.

We consider that dispersal between patches is fast with respect to demography
and we denote by $k$ the ratio between the characteristic time scales of both processes.

The state of the population at time $t$ is represented by vector
\[
X(t):=(X_{1}(t),\ldots,X_{q}(t))^{\mathsf{T}}\in\mathbb{R}_{+}^{qr},
\]
where $X_{i}(t):=(x_{i}^{1}(t),\ldots,x_{i}^{r}(t))\in\mathbb{R}_{+}^{r}$ and
$x_{i}^{\alpha}(t)$ denotes the population density in stage $i$ and patch
$\alpha$ at time $t$.

For each stage $i$, we represent dispersal by a projection matrix which is a
primitive probability matrix $M_{i}(Y)\in\mathbb{R}_{+}^{r\times r}$ possibly
depending on the total number of individuals in each stage. i.e., on vector
\[
Y:=\left(  y_{1},...,y_{q}\right)  ^{\mathsf{T}},
\]
where $y_{i}:=\sum_{\alpha=1}^{r}x_{i}^{\alpha}=\mathbf{1}X_{i}$ is the total
population in stage $i$ and we are denoting $\mathbf{1}:=(1,\overset{(r)}%
{...},1)$. Vector $Y$ will play the role of global variables \eqref{gobvar} in
the reduced system. If we denote $\mathbf{U}:=\text{diag}\left(
\mathbf{1},...,\mathbf{1}\right)  \in\mathbb{R}_{+}^{q\times qr}$ we can
express $Y$ in terms of $X$ in the following way:
\[
G(X):=Y=\mathbf{U}X.
\]
Finally, we define $\mathbf{M}(Y):=\text{diag}\left(  M_{1}(Y),\ldots
,M_{q}(Y)\right)  $ and thus the map $F$ defining the fast process, i.e.,
dispersal, is
\[
F(X):=\mathbf{M}(Y)X=\mathbf{M}(\mathbf{U}X)X.
\]
The fact that vector $\mathbf{1}$ is a left eigenvector of matrices $M_{i}$
associated to eigenvalue 1 implies that $\mathbf{U}\mathbf{M}(\mathbf{U}%
X)=\mathbf{U}$, so $\mathbf{U}\mathbf{M}(\mathbf{U}X)X=\mathbf{U}X$ and
\[
F^{(k)}(X)=\mathbf{M}(\mathbf{U}X)^{k}X,
\]
that is, we can express the $k$-th iterate of map $F$ in terms of the power
$k$ of matrix $\mathbf{M}(\mathbf{U}X)$.

In the sequel we assume that all the vital rates, and therefore the resulting
vital rate matrices, are $C^{1}$ functions of their corresponding variables.

To represent the slow process, associated to demography, we define a
nonnegative projection matrix that can depend on the state variables
\[
\mathbf{D}(X)=\left[  D_{ij}(X)\right]  _{1\leq i,j\leq q}\in\mathbb{R}%
^{qr\times qr},
\]
and is divided into blocks $D_{ij}(X)=\text{diag}\left(  d_{ij}^{\alpha
}(X)\right)  \in\mathbb{R}^{r\times r}$ where $d_{ij}^{\alpha}(X)$ represents
the rate of individual transition from stage $j$ to stage $i$ in patch
$\alpha$ during a slow time interval when the state of the population is
represented by vector $X$. Thus, the map defining the slow process is
\[
S(X):=\mathbf{D}(X)X.
\]

With the maps $F$ and $S$ we propose a first two time scales model in the form
of system \eqref{1-sg}:
\begin{equation}
X(t+1)=S\left(  F^{(k)}\left(  X(t)\right)  \right)  =\mathbf{D}\left(
\mathbf{M}(\mathbf{U}X(t))^{k}X(t)\right)  \mathbf{M}(\mathbf{U}X(t))^{k}X(t).
\label{mod1}%
\end{equation}

In this model individuals first perform a series of $k$ dispersal events
followed by a demographic update that occurs in the arrival patch. This
assumption is realistic concerning reproduction specially in the case of
populations with separated generations. However, the situation might be
different for a process like mortality, for individuals may die at any time
during the dispersal process. Now we show how to take this into account by
proposing a new model, based on the previous one, in which mortality acts at
the fast time scale, i.e., we re-scale survival to the fast time scale.

Let $s_{i}^{\alpha}>0$ be the stage and patch specific survival rate for stage
i and patch $\alpha$, that is, the fraction of individuals in stage $i$
($i=1,...,q$) alive at time $t$ that survive to time $t+1$ in patch $\alpha$
($\alpha=1,...,r$). We assume that $s_{i}^{\alpha}$ possibly depends on $Y$,
i.e, $s_{i}^{\alpha}=s_{i}^{\alpha}(Y)$.

We first factor out every coefficient $d_{ij}^{\alpha}(X)$ in matrix
$\mathbf{D}(X)$ in order to make the survival rate appear explicitly, i.e., we
define
\[
\tilde{d}_{ij}^{\alpha}(X):=d_{ij}^{\alpha}(X)/\,s_{j}^{\alpha}(Y).
\]
This factorization can be extended to matrix $\mathbf{D}(X)$. Indeed, we
define
\[
\tilde{D}_{ij}(X)=\text{diag}\left(  \tilde{d}_{ij}^{\alpha}(X)\right)
_{\alpha=1,\ldots,r},\quad\tilde{\mathbf{D}}(X)=\left[  \tilde{D}%
_{ij}(X)\right]  _{1\leq i,j\leq q},
\]%
\[
S_{j}(Y)=\text{diag}\left(  s_{j}^{\alpha}(Y)\right)  _{\alpha=1,\ldots
,r}\text{ and }\mathbf{S}(Y)=\text{diag}\left(  {S}_{j}(Y)\right)
_{j=1,\ldots,q},
\]
so that $D_{ij}(X)=\tilde{D}_{ij}(X)\,S_{j}(Y)$ and
\[
\mathbf{D}(X)=\tilde{\mathbf{D}}(X)\,\mathbf{S}(Y).
\]

We can now write $\mathbf{S}(Y)=\mathbf{S}_{k}(Y)^{k}$, i.e., $\mathbf{S}(Y)$
is the $k$-th power of the matrix $\mathbf{S}_{k}(Y)$ given by
\[
\mathbf{S}_{k}(Y)=\text{diag}\left(  {S}_{j,k}(Y)\right)  _{j=1,\ldots
,q}\text{ and }S_{j,k}(Y)=\text{diag}\left(  s_{j}^{\alpha}(Y)^{1/k}\right)
_{\alpha=1,\ldots,r}.
\]
We consider that matrix $\mathbf{S}_{k}(Y)$ describes the effect of survival
at the fast time scale, and based on it we propose a second two time scales
model with re-scaled survival:
\begin{equation}
X(t+1)=\tilde{\mathbf{D}}\left(  \left(  \mathbf{S}_{k}(\mathbf{U}%
X(t))\,\mathbf{M}(\mathbf{U}X(t))\rule{0ex}{2ex}\right)  ^{k}X(t)\right)
\left(  \mathbf{S}_{k}(\mathbf{U}X(t))\,\mathbf{M}(\mathbf{U}%
X(t))\rule{0ex}{2ex}\right)  ^{k}X(t). \label{mod2}%
\end{equation}

In this model individuals first perform a series of $k$ dispersal events in
which mortality is taken into account in each of them through corresponding
survival rate. The rest of the demographic process follows at the slow time scale.

\subsection{Reduction of models \eqref{mod1} and \eqref{mod2}}

\label{sec31}

To reduce \eqref{mod1} we show how the corresponding sequence of maps
\[
H_{k}(X)=\mathbf{D}\left(  \mathbf{M}(\mathbf{U}X)^{k}X\right)  \mathbf{M}%
(\mathbf{U}X)^{k}X
\]
satisfies Hypotheses \ref{H1}, \ref{H2}, and \ref{H3}.

The key point is to find the limit of the powers of matrix $\mathbf{M}%
(\mathbf{U}X)$. We use the fact that the $M_{i}(\mathbf{U}X)$ are primitive
probability matrices. Therefore their dominant eigenvalue is 1, vector
$\mathbf{1}$ is an associated left eigenvector and there exists, for each
$i=1,\ldots,q$ and $\mathbf{U}X\in\mathbb{R}_{+}^{q}$, a unique column
positive right eigenvector $V_{i}(\mathbf{U}X)\in\mathbb{R}^{r}$ such that
$\mathbf{1}V_{i}(\mathbf{U}X)=1$. The Perron-Frobenius theorem yields that
\[
\lim_{k\rightarrow\infty}M_{i}(\mathbf{U}X)^{k}=V_{i}(\mathbf{U}%
X))\mathbf{1}.
\]
Calling $\mathbf{V}(\mathbf{U}X):=\text{diag}\left(  V_{i}(\mathbf{U}%
X))\right)  _{i=1,\ldots,q}\in\mathbb{R}_{+}^{qr\times q}$ we have that
\[
\lim_{k\rightarrow\infty}\mathbf{M}(\mathbf{U}X)^{k}=\mathbf{V}(\mathbf{U}%
X)\,\mathbf{U}%
\]
and thus
\[
\lim_{k\rightarrow\infty}H_{k}(X)=\mathbf{D}\left(  \mathbf{V}(\mathbf{U}%
X)\,\mathbf{U}X\right)  \mathbf{V}(\mathbf{U}X)\,\mathbf{U}X,
\]
and so the limit in Hypothesis \ref{H1} exists. The decomposition of this
limit required in Hypothesis \ref{H2} is obtained by defining $G(X)=\mathbf{U}%
X$ and $T(Y)=\mathbf{D}\left(  \mathbf{V}(Y)Y\right)  \mathbf{V}(Y)Y$. The
rest of the technical details involved in proving that Hypotheses \ref{H1} and
\ref{H3} hold can be found in \cite{Marva09}.

The reduced system associated to \eqref{mod1} is
\begin{equation}
\label{mod1a}Y(t+1)=\bar{H}(Y(t))=G(T(Y(t)))=\mathbf{U}\,\mathbf{D}\left(
\mathbf{V}(Y(t))Y(t)\right)  \mathbf{V}(Y(t))Y(t).
\end{equation}
Theorem \ref{th1} can be applied to systems \eqref{mod1} and \eqref{mod1a} so
that we can obtain information on the asymptotic behaviour of solutions to
system \eqref{mod1} by performing the analysis in the reduced system \eqref{mod1a}.

We now proceed to the reduction of system \eqref{mod2}, for which we have
\[
H_{k}(X)=\tilde{\mathbf{D}}\left(  \left(  \mathbf{S}_{k}(\mathbf{U}%
X)\,\mathbf{M}(\mathbf{U}X)\rule{0ex}{2ex}\right)  ^{k}X\right)  \left(
\mathbf{S}_{k}(\mathbf{U}X)\,\mathbf{M}(\mathbf{U}X)\rule{0ex}{2ex}\right)
^{k}X.
\]
In this occasion we have to calculate the limit of the $k$-th power of matrix
$\mathbf{S}_{k}(\mathbf{U}X)\,\mathbf{M}(\mathbf{U}X)$.

A straightforward application of Proposition \ref{prop:conv} in the Appendix
to each of the diagonal blocks of matrix $\mathbf{S}_{k}(\mathbf{U}%
X)\,\mathbf{M}(\mathbf{U}X)$ yields the result. For every $i=1,\ldots,q$ let
$Z_{i}(\mathbf{U}X)=\left(  \log(s_{i}^{1}(\mathbf{U}X)),\ldots,\log(s_{i}%
^{r}(\mathbf{U}X))\right)  \in\mathbb{R}^{r}$ be a row vector and define the
scalar $\gamma_{i}(\mathbf{U}X):=\exp\left(  Z_{i}(\mathbf{U}X)\,V_{i}%
(\mathbf{U}X)\right)  $. Denoting
\[
\tilde{\mathbf{V}}:=\text{diag}\left(  \gamma_{i}(\mathbf{U}X)V_{i}%
(\mathbf{U}X)\right)  _{i=1,\ldots,q}\in\mathbb{R}_{+}^{q\times qr},
\]
we obtain
\[
\lim_{k\rightarrow\infty}\left(  \mathbf{S}_{k}(\mathbf{U}X)\,\mathbf{M}%
(\mathbf{U}X)\right)  ^{k}=\tilde{\mathbf{V}}(\mathbf{U}X)\,\mathbf{U}.
\]
Therefore,
\[
\lim_{k\rightarrow\infty}H_{k}(X)=\tilde{\mathbf{D}}\left(  \tilde{\mathbf{V}%
}(\mathbf{U}X)\,\mathbf{U}X\right)  \tilde{\mathbf{V}}(\mathbf{U}%
X)\,\mathbf{U}X.
\]
Defining $G(X)=\mathbf{U}X$ and $\tilde{T}(Y)=\tilde{\mathbf{D}}\left(
\tilde{\mathbf{V}}(Y)Y\right)  \tilde{\mathbf{V}}(Y)Y,$ Hypotheses \ref{H1},
\ref{H2} and \ref{H3} are met and the reduced system associated to
\eqref{mod2} is
\begin{equation}
Y(t+1)=\tilde{H}(Y(t))=G(\tilde{T}(Y))=\mathbf{U}\,\tilde{\mathbf{D}}\left(
\tilde{\mathbf{V}}(Y(t))Y(t)\right)  \tilde{\mathbf{V}}(Y(t))Y(t).
\label{mod2a}%
\end{equation}

\section{Models \eqref{mod1} and \eqref{mod2} in a three-stage, two-patch
case}

\label{sec4}

We propose the model of a population inhabiting a two-patch environment.
Following the model presented in \cite{Veprauskas17}, locally the population
is considered structured into three classes corresponding to juveniles and two
adult stages. Adults that reproduce at the end of an interval of time do not
reproduce at the end of the following one. A fraction of those not having
reproduced do at the end of the following interval. Thus, adults are
classified into those reproducing at the end of the time interval, active
adults, and those adults who do not, inactive adults. The aim of this
consideration is analyzing the reproductive synchrony of adults.

Let $x_{1}^{\alpha}$, $x_{2}^{\alpha}$ and $x_{3}^{\alpha}$ be the number of
juveniles, active adults and inactive adults, respectively, in patch
$\alpha=1,2$. The demographic vector is then
\[
X=\left(  x_{1}^{1},x_{1}^{2},x_{2}^{1},x_{2}^{2},x_{3}^{1},x_{3}^{2}\right)
^{\mathsf{T}}\in\mathbb{R}_{+}^{6}.
\]
The unit of time of the discrete model that we are proposing coincides with
the juvenile maturation period. For $i=1,2,3$ and $\alpha=1,2$, we denote
$s_{i}^{\alpha}\in(0,1)$ the corresponding survival probabilities, $f^{\alpha
}>0$ the adult per capita fecundity rates and $g^{\alpha}\in[0,1]$ the
fractions of reproductively inactive adults becoming active at the next time
unit. Figure \ref{figgraph} shows a transition graph for the population.

\setlength{\unitlength}{2ex}
\begin{figure}[ht]
\begin{center}
\noindent
\begin{picture}(36,20)

\put(19,15){\oval(33.5,8)}
\put(19,5){\oval(33.5,8)}

\put(7,14){\oval(6,4)}
\put(17,14){\oval(6,4)}
\put(27,14){\oval(6,4)}
\put(7,6){\oval(6,4)}
\put(17,6){\oval(6,4)}
\put(27,6){\oval(6,4)}

\put(18,17){\makebox(0,0)[b]{\footnotesize{\textbf{PATCH 1}}}}
\put(18,2.5){\makebox(0,0)[b]{\footnotesize{\textbf{PATCH 2}}}}

\put(6.5,8.2){\vector(0,1){3.6}}
\put(7.5,11.8){\vector(0,-1){3.6}}
\put(16.5,8.2){\vector(0,1){3.6}}
\put(17.5,11.8){\vector(0,-1){3.6}}
\put(26.5,8.2){\vector(0,1){3.6}}
\put(27.5,11.8){\vector(0,-1){3.6}}

\put(10.2,14.5){\vector(1,0){3.6}}
\put(13.8,13.5){\vector(-1,0){3.6}}
\put(20.2,14.5){\vector(1,0){3.6}}
\put(23.8,13.5){\vector(-1,0){3.6}}
\put(10.2,6.5){\vector(1,0){3.6}}
\put(13.8,5.5){\vector(-1,0){3.6}}
\put(20.2,6.5){\vector(1,0){3.6}}
\put(23.8,5.5){\vector(-1,0){3.6}}

\put(29.7,14){\oval(3,3)[tr]}
\put(29.7,14){\oval(3,3)[br]}
\put(29.7,12.5){\vector(-1,0){0.1}}
\put(29.7,6){\oval(3,3)[tr]}
\put(29.7,6){\oval(3,3)[br]}
\put(29.7,7.5){\vector(-1,0){0.1}}

\put(7,14.7){\makebox(0,0){\scriptsize{\textbf{Juveniles}}}}
\put(7,13.1){\makebox(0,0){$x_1^1$}}
\put(17,15.3){\makebox(0,0){\scriptsize{\textbf{Active}}}}
\put(17,14.5){\makebox(0,0){\scriptsize{\textbf{adults}}}}
\put(17,13.1){\makebox(0,0){$x_2^1$}}
\put(27,15.3){\makebox(0,0){\scriptsize{\textbf{Inactive}}}}
\put(27,14.5){\makebox(0,0){\scriptsize{\textbf{adults}}}}
\put(27,13.1){\makebox(0,0){$x_3^1$}}
\put(7,6.7){\makebox(0,0){\scriptsize{\textbf{Juveniles}}}}
\put(7,5.1){\makebox(0,0){$x_1^2$}}
\put(17,7.3){\makebox(0,0){\scriptsize{\textbf{Active}}}}
\put(17,6.5){\makebox(0,0){\scriptsize{\textbf{adults}}}}
\put(17,5.1){\makebox(0,0){$x_2^2$}}
\put(27,7.3){\makebox(0,0){\scriptsize{\textbf{Inactive}}}}
\put(27,6.5){\makebox(0,0){\scriptsize{\textbf{adults}}}}
\put(27,5.1){\makebox(0,0){$x_3^2$}}

\put(12,15.3){\makebox(0,0){$s_1^1$}}
\put(12,12.6){\makebox(0,0){$f^1 s_2^1$}}
\put(22,15.3){\makebox(0,0){$s_2^1$}}
\put(22,12.6){\makebox(0,0){$g^1 s_3^1$}}
\put(33.3,14){\makebox(0,0){$1-g^1 s_3^1$}}
\put(12,7.3){\makebox(0,0){$s_1^2$}}
\put(12,4.6){\makebox(0,0){$f^2 s_2^2$}}
\put(22,7.3){\makebox(0,0){$s_2^2$}}
\put(22,4.6){\makebox(0,0){$g^2 s_3^2$}}
\put(33.3,6){\makebox(0,0){$1-g^2 s_3^2$}}
\end{picture}
\caption{Transition graph of a population structured in three classes and two patches.}
\label{figgraph}
\end{center}
\end{figure}
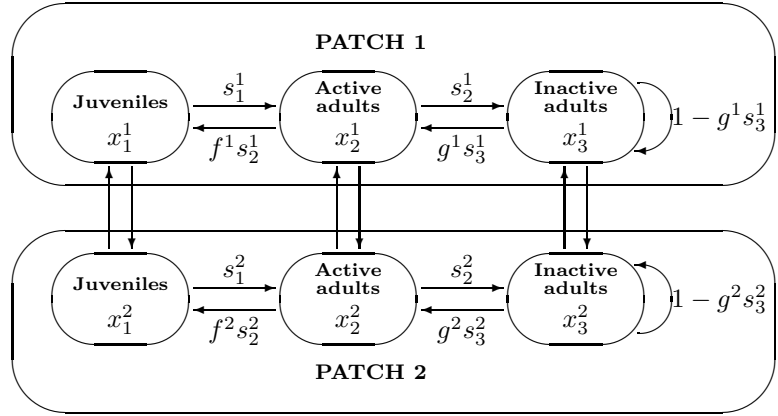

The projection matrix, $\mathbf{D}$, associated to demography is
\[
\mathbf{D}=\left(
\begin{array}
[c]{cccccc}%
0 & 0 & f^{1}s_{2}^{1} & 0 & 0 & 0\\
0 & 0 & 0 & f^{2}s_{2}^{2} & 0 & 0\\
s_{1}^{1} & 0 & 0 & 0 & g^{1}s_{3}^{1} & 0\\
0 & s_{1}^{2} & 0 & 0 & 0 & g^{2}s_{3}^{2}\\
0 & 0 & s_{2}^{1} & 0 & (1-g^{1})s_{3}^{1} & 0\\
0 & 0 & 0 & s_{2}^{2} & 0 & (1-g^{2})s_{3}^{2}%
\end{array}
\right)
\]
that can also be written as
\[
\mathbf{D}=\tilde{\mathbf{D}}\mathbf{S}=\left(
\begin{array}
[c]{cccccc}%
0 & 0 & f^{1} & 0 & 0 & 0\\
0 & 0 & 0 & f^{2} & 0 & 0\\
1 & 0 & 0 & 0 & g^{1} & 0\\
0 & 1 & 0 & 0 & 0 & g^{2}\\
0 & 0 & 1 & 0 & 1-g^{1} & 0\\
0 & 0 & 0 & 1 & 0 & 1-g^{2}%
\end{array}
\right)  \left(
\begin{array}
[c]{cccccc}%
s_{1}^{1} & 0 & 0 & 0 & 0 & 0\\
0 & s_{1}^{2} & 0 & 0 & 0 & 0\\
0 & 0 & s_{2}^{1} & 0 & 0 & 0\\
0 & 0 & 0 & s_{2}^{2} & 0 & 0\\
0 & 0 & 0 & 0 & s_{3}^{1} & 0\\
0 & 0 & 0 & 0 & 0 & s_{3}^{2}%
\end{array}
\right)
\]
Migrations are represented by primitive probability matrices $M_{i}%
\in\mathbb{R}_{+}^{2\times2}$ ($i=1,2,3$), that are the diagonal blocks of
\[
\mathbf{M}=\text{diag}\left(  M_{1},M_{2},M_{3}\right)  .
\]
According to the framework of models \eqref{mod1} and \eqref{mod2},
fecundities and transitions from reproductive inactivity may depend on the
state variables, whereas the survival and dispersal coefficients could
possibly depend on the total number of individuals in each stage
\[
Y:=\left(  y_{1},y_{2},y_{3}\right)  ^{\mathsf{T}},
\]
where $y_{i}=x_{i}^{1}+x_{i}^{2}=\mathbf{1}X_{i}$, with $\mathbf{1}=(1,1)$, is
the total population in stage $i$. We recall that in the present case
$\mathbf{U}=\text{diag}\left(  \mathbf{1},\mathbf{1},\mathbf{1}\right)
\in\mathbb{R}_{+}^{3\times6}$, and we have $Y=\mathbf{U}X$.

The model with survival acting at the slow time scale is
\begin{equation}
X(t+1)=\mathbf{D}\left(  \mathbf{M}(Y(t))^{k}X(t)\right)  \mathbf{M}%
(Y(t))^{k}X(t), \label{mod41}%
\end{equation}
whereas the model with survival acting at the fast time scale is
\begin{equation}
X(t+1)=\tilde{\mathbf{D}}\left(  \left(  \mathbf{S}_{k}(Y(t))\,\mathbf{M}%
(Y(t))\rule{0ex}{2ex}\right)  ^{k}X(t)\right)  \left(  \mathbf{S}%
_{k}(Y(t))\,\mathbf{M}(Y(t))\rule{0ex}{2ex}\right)  ^{k}X(t), \label{mod42}%
\end{equation}
where $\mathbf{S}_{k}=\text{diag}\left(  S_{k1},S_{k2},S_{k3}\right)  $ with
$S_{ki}=\text{diag}\left(  (s_{i}^{1})^{1/k},(s_{i}^{2})^{1/k}\right)  $ for
$i=1,2,3$.

To simplify calculations we first assume that dispersal rates are constant.

For $i=1,2,3$ let $V_{i}=(v_{i}^{1},v_{i}^{2})^{\mathsf{T}}$ be the unique
column positive right eigenvector of matrix $M_{i}$ associated to eigenvalue 1
such that $\mathbf{1}V_{i}=1$. We are using matrix
\[
\mathbf{V}=\text{diag}\left(  V_{1},V_{2},V_{3}\right)  ,
\]
to obtain the reduced systems \eqref{mod1a} and \eqref{mod2a} associated to
systems \eqref{mod1} and \eqref{mod2}.

We also assume, following the example in Section 2 of \cite{Veprauskas17},
that survival rates are constant and fecundities and transitions from
reproductive inactivity are locally dependent on the number of active adults.
Specifically, we set
\[
f^{\alpha}(x_{2}^{\alpha}):=\frac{\phi^{\alpha}}{1+c^{\alpha}x_{2}^{\alpha}%
}\ \text{ and }g^{\alpha}(x_{2}^{\alpha}):=\frac{1}{1+d^{\alpha}x_{2}^{\alpha
}},\ \alpha=1,2,
\]
where $\phi^{\alpha}>0$ are the inherent fertility rates and parameters
$c^{\alpha}$ and $d^{\alpha}$ are positive.

To obtain the reduced system associated to system \eqref{mod1} we apply the
procedure described in Section \ref{sec31}. The result is the next
3-dimensional system:
\begin{equation}
Y(t+1)=\left(
\begin{array}
[c]{ccc}%
0 & \bar{b}\bar{h}_{1}(y_{2}(t)) & 0\\
\bar{s}_{1} & 0 & \bar{s}_{3}\bar{h}_{2}(y_{2}(t))\\
0 & \bar{s}_{2} & \bar{s}_{3}\left(  1-\bar{h}_{2}(y_{2}(t))\right)
\end{array}
\right)  Y(t)=\bar{H}(Y(t))Y(t), \label{mod41a}%
\end{equation}
where
\[
\bar{s}_{i}:=s_{i}^{1}v_{i}^{1}+s_{i}^{2}v_{i}^{2}\ (i=1,2,3)\text{ , }\bar
{b}:=\phi^{1}s_{2}^{1}v_{2}^{1}+\phi^{2}s_{2}^{2}v_{2}^{2}\ ,
\]%
\[
\bar{h}_{1}(y_{2}):=\frac{\phi^{1}s_{2}^{1}v_{2}^{1}/\bar{b}}{1+c^{1}v_{2}%
^{1}y_{2}}+\frac{\phi^{2}s_{2}^{2}v_{2}^{2}/\bar{b}}{1+c^{2}v_{2}^{2}y_{2}%
}\ \text{ and }\ \bar{h}_{2}(y_{2}):=\frac{s_{3}^{1}v_{3}^{1}/\bar{s}_{3}%
}{1+d^{1}v_{2}^{1}y_{2}}+\frac{s_{3}^{2}v_{3}^{2}/\bar{s}_{3}}{1+d^{2}%
v_{2}^{2}y_{2}}.
\]

Associated to system \eqref{mod2} we obtain the following reduced system:
\begin{equation}
Y(t+1)=\left(
\begin{array}
[c]{ccc}%
0 & \tilde{b}\tilde{h}_{1}(y_{2}(t)) & 0\\
\tilde{s}_{1} & 0 & \tilde{s}_{3}\tilde{h}_{2}(y_{2}(t))\\
0 & \tilde{s}_{2} & \tilde{s}_{3}\left(  1-\tilde{h}_{2}(y_{2}(t))\right)
\end{array}
\right)  Y(t)=\tilde{H}(Y(t))Y(t), \label{mod42a}%
\end{equation}
where
\[
\tilde{s}_{i}:=\left(  s_{i}^{1}\right)  ^{v_{i}^{1}}\left(  s_{i}^{2}\right)
^{v_{i}^{2}}\ (i=1,2,3)\text{ , }\tilde{b}:=\tilde{s}_{2}\left(  \phi^{1}%
v_{2}^{1}+\phi^{2}v_{2}^{2}\right)  \ ,
\]%
\[
\tilde{h}_{1}(y_{2}):=\frac{\phi^{1}\tilde{s}_{2}v_{2}^{1}/\tilde{b}}%
{1+c^{1}\tilde{s}_{2}v_{2}^{1}y_{2}}+\frac{\phi^{2}\tilde{s}_{2}v_{2}%
^{2}/\tilde{b}}{1+c^{2}\tilde{s}_{2}v_{2}^{2}y_{2}}\ \text{ and }\ \tilde
{h}_{2}(y_{2}):=\frac{v_{3}^{1}}{1+d^{1}\tilde{s}_{2}v_{2}^{1}y_{2}}%
+\frac{v_{3}^{2}}{1+d^{2}\tilde{s}_{2}v_{2}^{2}y_{2}}.
\]

The projection matrices, $\bar{H}(Y)$ and $\tilde{H}(Y)$, of systems
\eqref{mod41a} and \eqref{mod42a} verify the hypotheses H1 and H2 of
projection matrix $P$ in Section 3 of \cite{Veprauskas17}, so Theorems 1-3
therein apply to both systems.

We now proceed to adapt the results in Section 3 of \cite{Veprauskas17} to
systems \eqref{mod41a} and \eqref{mod42a}. Thus, definitions, notation and
propositions are directly inspired from it.

In the first place we define the \textit{inherent net reproduction number}
$R_{0}$ associated to a non-linear system as the net reproduction number (NRN)
\cite{cushing1994net} of the projection matrix of the system in the absence of
density dependence, i.e., when the population vector is zero. In this way, the
inherent NRNs, $\bar{R}_{0}$ and $\tilde{R}_{0}$, of matrices $\bar{H}(Y)$ and
$\tilde{H}(Y)$ are:
\begin{equation}
\bar{R}_{0}=\frac{\bar{b}\bar{s}_{1}}{1-\bar{s}_{2}\bar{s}_{3}}\ \text{ and
}\ \tilde{R}_{0}=\frac{\tilde{b}\tilde{s}_{1}}{1-\tilde{s}_{2}\tilde{s}_{3}}.
\label{r0s}%
\end{equation}

The first result states that the inherent NRN characterizes the stability of
the extinction equilibrium of the system as well as its uniform persistence.

\begin{proposition}
\label{pr1} For system \eqref{mod41a} (resp. \eqref{mod42a}) the extinction
equilibrium $Y_{0}^{\ast}=\bar{0}=(0,0,0)^{\mathsf{T}}$ is locally
asymptotically stable if $\bar{R}_{0}<1$ (resp. $\tilde{R}_{0}<1$) and
unstable if $\bar{R}_{0}>1$ (resp. $\tilde{R}_{0}>1$). Furthermore, if
$\bar{R}_{0}>1$ (resp. $\tilde{R}_{0}>1$) then the system \eqref{mod41a}
(resp. \eqref{mod42a}) is uniformly persistent with respect to $Y_{0}^{\ast}$.
\end{proposition}

\begin{proof}
The result is a direct consequence of Theorem 1 in \cite{Veprauskas17}.
\end{proof}

In the sequel we concentrate on system \eqref{mod41a} and we develop some
asymptotic results when $\bar{R}_{0}$ increases through 1. That entails, as
shown in Proposition \ref{pr1}, the destabilization of the extinction
equilibrium $Y_{0}^{\ast}=\bar{0}$. The fact that matrix $\bar{H}(\bar{0})$ is
not primitive gives rise to the simultaneous bifurcation of positive
equilibria and non-negative 2-cycles \cite{Cushing15}.

We \ say that $(\bar{R}_{0},\bar{Y})\in\mathbf{R}_{+}\times\mathbf{R}_{+}^{3}$
is an equilibrium pair for system \eqref{mod41a} when $\bar{Y}$ is an
equilibrium for the system and the inherent NRN is $\bar{R}_{0}$. Clearly
$(\bar{R}_{0},\bar{0})$ is an equilibrium pair for all values of $\bar{R}_{0}$.

Clearly, system \eqref{mod41a} has periodic orbits on the boundary of the
positive cone of the following form
\[
\bar{Y}_{2}=(0,y_{2},0)^{\mathrm{T}},\bar{Y}_{1,3}=(y_{1},0,y_{3}%
)^{\mathrm{T}}%
\]
for certain values of $y_{1},y_{2},y_{3}>0$. They are called synchronous
2-cycles and are represented by their two points $(\bar{Y}_{2},\bar{Y}_{1,3}%
)$. We call $\left(  \bar{R}_{0},(\bar{Y}_{2},\bar{Y}_{1,3})\right)  $ a
synchronous 2-cycle pair of system \eqref{mod41a} if $(\bar{Y}_{2},\bar
{Y}_{1,3})$ is a synchronous 2-cycle for the associated value of $\bar{R}_{0}$.

The declared aim of the model in \cite{Veprauskas17} is analyzing the
reproductive synchrony of adults. Concerning that, the synchronous 2-cycles
represent reproductive synchrony, i.e., all adults reproduce simultaneously in
only one of the two points of the cycle, whereas the positive equilibria
represent reproductive asynchrony, i.e., there are reproducing adults at each
point of time. Through models \eqref{mod41} and \eqref{mod42} we can study how
dispersal affect reproductive synchrony. At the same time, we are interested
in analyzing whether choosing one model versus the other could result in
different outcomes.

In the next result, conditions for the existence and stability of equilibrium
and synchronous 2-cycle pairs are obtained. To do so, we define the following
four quantities
\begin{equation}
\bar{c}_{w}:=(1-\bar{s}_{2}\bar{s}_{3})\bar{s}_{1}\bar{h}_{1}^{\prime
}(0),\ \bar{c}_{b}:=\bar{s}_{1}\bar{s}_{2}\bar{s}_{3}(1-\bar{s}_{3})\bar
{h}_{2}^{\prime}(0) \label{pr2-1}%
\end{equation}
and
\begin{equation}
\bar{a}_{+}:=\bar{c}_{w}+\bar{c}_{b},\ \bar{a}_{-}:=\bar{c}_{w}-\bar{c}_{b}
\label{pr2-2}%
\end{equation}

\begin{proposition}
\label{pr2} For system \eqref{mod41a}:

\begin{enumerate}
\item A continuum $\bar{C}_{e}$ of positive equilibrium pairs bifurcates from
the extinction equilibrium pair $(\bar{R}_{0},\bar{Y})=(1,\bar{0})$. In a
neighbourhood of $(1,\bar{0})$, the positive equilibrium pairs $(\bar{R}%
_{0},\bar{Y})\in\bar{C}_{e}$ have, for $0<\varepsilon\ll1$, the
parameterizations
\[
\bar{R}_{0}(\varepsilon)=1-\frac{\bar{a}_{+}}{1-\bar{s}_{2}\bar{s}_{3}%
}\varepsilon+O(\varepsilon^{2}),\ \bar{Y}(\varepsilon)=\left(
\begin{array}
[c]{c}%
1-\bar{s}_{2}\bar{s}_{3}\\
\bar{s}_{1}\\
\bar{s}_{1}\bar{s}_{2}\\
\end{array}
\right)  \varepsilon+O(\varepsilon^{2}).
\]

\item A continuum $\bar{C}_{2}$ of synchronous 2-cycles pairs bifurcates from
the extinction equilibrium pair $(\bar{R}_{0},\bar{Y})=(1,\bar{0})$. In a
neighbourhood of $(1,\bar{0})$, the synchronous 2-cycles pairs $(\bar{R}%
_{0},(\bar{Y}_{2},\bar{Y}_{1,3}))\in\bar{C}_{2}$ have, for $0<\varepsilon\ll
1$, the parameterizations
\[
\bar{R}_{0}(\varepsilon)=1-\frac{\bar{c}_{w}}{1-\bar{s}_{2}\bar{s}_{3}%
}\varepsilon+O(\varepsilon^{2}),
\]%
\[
\bar{Y}_{2}(\varepsilon)=\left(
\begin{array}
[c]{c}%
0\\
\bar{s}_{1}\\
0\\
\end{array}
\right)  \varepsilon+O(\varepsilon^{2}),\ \bar{Y}_{1,3}(\varepsilon)=\left(
\begin{array}
[c]{c}%
1-\bar{s}_{2}\bar{s}_{3}\\
0\\
\bar{s}_{1}\bar{s}_{2}\\
\end{array}
\right)  \varepsilon+O(\varepsilon^{2}).
\]

\item
\begin{enumerate}
\item If $\bar{a}_{-}<0$ the equilibria of the pairs in $\bar{C}_{e}$ are
locally asymptotically stable and the 2-cycles of the pairs in $\bar{C}_{2}$
are unstable.

\item If $\bar{a}_{-}>0$ the equilibria of the pairs in $\bar{C}_{e}$ are
unstable and the 2-cycles of the pairs in $\bar{C}_{2}$ are locally
asymptotically stable.
\end{enumerate}
\end{enumerate}
\end{proposition}

\begin{proof}
It is a direct consequence of Theorems 2 and 3 in \cite{Veprauskas17}.
\end{proof}

The previous result refers to system \eqref{mod41a}. Since systems
\eqref{mod41a} and \eqref{mod42a} have the same functional form, an analogous
result holds for the latter. In particular, it is the sign of the quantity
\begin{equation}
\tilde{a}_{-}:=(1-\tilde{s}_{2}\tilde{s}_{3})\tilde{s}_{1}\tilde{h}%
_{1}^{\prime}(0)-\tilde{s}_{1}\tilde{s}_{2}\tilde{s}_{3}(1-\tilde{s}%
_{3})\tilde{h}_{2}^{\prime}(0) \label{pr2-3}%
\end{equation}
which decides the asymptotic stability of either the equilibria of the pairs
in $\tilde{C}_{e}=(\tilde{R}_{0},\tilde{Y})$ ($\tilde{a}_{-}<0$) or the
synchronous 2-cycles of the pairs in $\tilde{C}_{2}=\left(  \tilde{R}%
_{0},(\tilde{Y}_{2},\tilde{Y}_{1,3})\right)  $ ($\tilde{a}_{-}>0$) that
bifurcate from the extinction equilibrium pair $(\tilde{R}_{0},\tilde
{Y})=(1,\bar{0})$.

The results in Sections \ref{sec2} and \ref{sec3} justify that the previous
results regarding the reduced systems \eqref{mod41a} and \eqref{mod42a} can be used to
obtain information about  the local dynamics of the original two time scales models
\eqref{mod41} and \eqref{mod42}.

In Figure \ref{fig10} a simulation regarding the relationship between systems
\eqref{mod42a} and \eqref{mod42} is shown, illustrating the fact that, for
large enough values of $k$, the reduced system \eqref{mod42a} is a good
approximation of the complete system \eqref{mod42}. Starting from the same
initial condition we have simulated system \eqref{mod42a} and system
\eqref{mod42} for three different values of $k$. The reduced system
\eqref{mod42a} presents a 2-cycle and so does the complete system
\eqref{mod42} for the three values of $k$. For $k=10$, the orbit of the
complete system is approximated very closely by that corresponding to the reduced system.


\begin{figure}[h]
\centering
\includegraphics[width=0.9\textwidth]{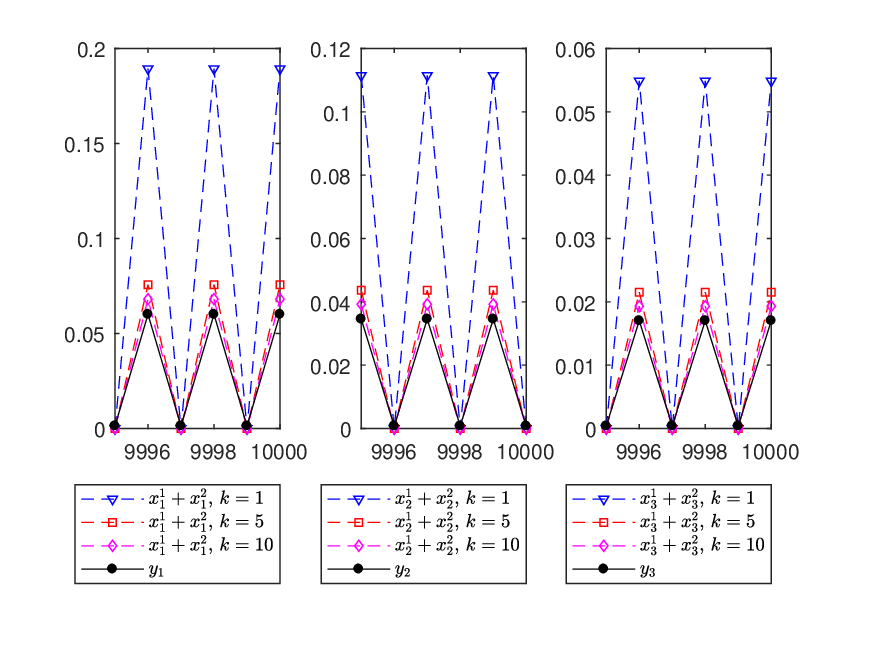}
\caption{Comparison between the complete and the reduced systems.
Total population corresponding to certain initial conditions for
system \eqref{mod42a} and for system \eqref{mod42} for three different values
of $k$, i.e., $k\in\{1,5,10\}$. The simulations have been run until time
$t=10^{4}$ and only the last 6 times are shown. \newline Parameter values:
$s_{1}^{1}=0.3$, $s_{2}^{1}=0.47$, $s_{3}^{1}=0.7$, $s_{1}^{2}=s_{2}^{2}%
=s_{3}^{2}=0.5$, $\phi^{1}=3.8$, $\phi^{2}=3.5$, $c_{1}=c_{2}=1$, $d_{1}=5.3$,
$d_{2}=5.6$, $v_{1}^{1}=0.3$, $v_{2}^{1}=0.25$, $v_{3}^{1}=0.1$. Initial
condition: $X(0)=(0.02,0.02,0.05,0.05,0.02,0.02)$. }
\label{fig10}
\end{figure}

\subsection{The effect of fast dispersal}

\label{sec41} The effect of dispersal in populations and in particular its
effect at a global level when local isolated populations are connected, has
considerable biological interest and has been addressed in a number of
different contexts (see
\cite{jang2000equilibrium,yakubu2008asynchronous,franco2015connect} among
others). In this section we deal with models (\ref{mod41}) and (\ref{mod42})
and present some results that show how dispersal can change the two main
asymptotic issues regarding the dynamics of the population. The first one is
the survival of the population as reflected by the extinction equilibrium
stability. Then, once the survival threshold is attained, we address the
second issue, i.e., the tendency of the population towards either reproductive
synchrony or asynchrony, represented by the stability of synchronous 2-cycles
or positive equilibria, respectively.

We compare the outcomes of local (isolated sites) and global (dispersal
interconnected sites) population dynamics. The reduction procedure applied to
systems \eqref{mod41} and \eqref{mod42} allows us to use the corresponding
reduced systems \eqref{mod41a} and \eqref{mod42a} to study their asymptotic
behaviour. So, we use quantities $\bar{R}_{0}$ (resp. $\tilde{R}_{0}$) and
$\bar{a}_{-}$ (resp. $\tilde{a}_{-}$) to characterize, respectively, survival
and reproductive synchrony at the global level in the two systems. The
influence of fast dispersal in the population dynamics is reflected through
the values $v_{i}^{\alpha}$ that appear in the coefficients of matrices
$\bar{H}(Y)$ \eqref{mod41a} and $\tilde{H}(Y)$ \eqref{mod42a}. We recall that
$V_{i}=(v_{i}^{1},v_{i}^{2})^{\mathsf{T}}$ ($i=1,2,3$) is the dominant
probability normed eigenvector of matrix $M_{i}$ and therefore its components
represent the equilibrium distribution of dispersal for individuals of class
$i$\textit{, }i.e, $v_{i}^{j}$ is the proportion of individuals of class $i$
present in patch $j$ after dispersal has reached equilibrium. Their value is%
\[
v_{i}^{1}=\frac{q_{i}}{p_{i}+q_{i}},\ v_{i}^{2}=\frac{p_{i}}{p_{i}+q_{i}},
\]
where $p_{i}$ (resp. $q_{i})$ is the migration rate from patch one to patch
two (resp. from patch two to patch one) for individuals of class $i$.

At the local level, i.e., if patches $\alpha=1,2$ are considered isolated, the
matrix representing the population dynamics in each of them is
\begin{equation}
\mathbf{D}^{\alpha}(Y)=\left(
\begin{array}
[c]{ccc}%
0 & s_{2}^{\alpha}\phi^{\alpha}\dfrac{1}{1+c^{\alpha}y_{2}^{\alpha}} & 0\\
s_{1}^{\alpha} & 0 & s_{3}^{\alpha}\dfrac{1}{1+d^{\alpha}y_{2}^{\alpha}}\\
0 & s_{2}^{\alpha} & s_{3}^{\alpha}\left(  1-\dfrac{1}{1+d^{\alpha}%
y_{2}^{\alpha}}\right) \\
&  &
\end{array}
\right)  \label{pr3-1}%
\end{equation}

\noindent\textbf{Effect of dispersal on extinction.} The stability of the
extinction equilibrium is locally determined by the corresponding inherent
NRR
\begin{equation}
R_{0}^{\alpha}:=\frac{\phi^{\alpha}s_{1}^{\alpha}s_{2}^{\alpha}}%
{1-s_{2}^{\alpha}s_{3}^{\alpha}},\ \alpha=1,2, \label{pr3-2}%
\end{equation}
and the tendency towards reproductive synchrony or asynchrony by
\begin{equation}
a_{-}^{\alpha}:=-(1-s_{2}^{\alpha}s_{3}^{\alpha})s_{1}^{\alpha}c^{\alpha
}+s_{1}^{\alpha}s_{2}^{\alpha}s_{3}^{\alpha}(1-s_{3}^{\alpha})d^{\alpha
},\ ~\alpha=1,2. \label{pr3-3}%
\end{equation}

If we consider homogeneous sites concerning survival and fertility, i.e.,
$s_{i}^{1}=s_{i}^{2}=:s_{i}$ for $i=1,2,3$ and $\phi^{1}=\phi^{2}=:\phi$, we
obtain that
\begin{equation}
R_{0}^{1}=R_{0}^{2}=\bar{R}_{0}=\tilde{R}_{0}=\frac{\phi s_{1}s_{2}}%
{1-s_{2}s_{3}}, \label{inh21}%
\end{equation}
so that, in particular, the population survives or gets extinct locally if and
only if it does globally.

On the other hand, differences in local parameters, i.e., $s_{1}^{1}$ vs.
$s_{1}^{2}$ and $\phi^{1}$ vs. $\phi^{2},$ can result in different survival
outcomes at the local and the global levels provided appropriate dispersal
rates are chosen. In the next result we show that under certain conditions,
adequate dispersal strategies can transform the local population extinction in
isolated patches into global survival.

\begin{proposition}
\label{pr3} Let us assume $s_{i}^{1}=s_{i}^{2}=:s_{i}$ for $i=2,3$, $R_{0}%
^{1}<1$ and $R_{0}^{2}<1$. If $\max\{s_{1}^{1}\phi^{2},s_{1}^{2}\phi
^{1}\}>(1-s_{2}s_{3})/s_{2}$ then there exist intervals $I_{1},I_{2}%
\subseteq\lbrack0,1]$ (resp. $\tilde{I}_{1},\tilde{I}_{2}\subseteq\lbrack
0,1]$) such that $\bar{R}_{0}>1$ (resp. $\tilde{R}_{0}>1$) for $v_{1}^{1}\in
I_{1}$ and $v_{2}^{1}\in I_{2}$ (resp. $v_{1}^{1}\in\tilde{I}_{1}$ and
$v_{2}^{1}\in\tilde{I}_{2}$).
\end{proposition}

\begin{proof}
Let us assume, without loss of generality, that $s^{1}_{1} \phi^{2} \geq
s^{2}_{1}\phi^{1}$ and use that $v_{i}^{2} = 1-v_{i}^{1}$ for $i=1,2,3$ to
obtain
\[%
\begin{array}
[c]{rcl}%
\bar{R}_{0} & = & \dfrac{s_{2}}{1-s_{2}s_{3}}\left(  v_{1}^{1}s^{1}%
_{1}+(1-v_{1}^{1})s^{2}_{1}\right)  \left(  v_{2}^{1}\phi^{1}+(1-v_{2}%
^{1})\phi^{2}\right) \\
& > & \dfrac{s_{2}}{1-s_{2}s_{3}}v_{1}^{1}(1-v_{2}^{1})s^{1}_{1}\phi^{2}.
\end{array}
\]
This last expression considered as function of variables $v_{1}^{1}$ and
$v_{2}^{1}$ is continuous and takes the value $s^{1}_{1}\phi^{2}s_{2}%
/(1-s_{2}s_{3})>1$ for $v_{1}^{1}=1 $ and $v_{2}^{1}=0$. This yields the
existence of the intervals $I_{1},I_{2}\subseteq[0,1]$ that ensure $\bar
{R}_{0}>1$ for $(v_{1}^{1},v_{2}^{1})\in I_{1}\times I_{2}$.

The proof for $\tilde{R}_{0}>1$ is analogous.
\end{proof}

The previous result says that given certain (common) adult survival rates, if
the juvenile survival rate in one of the patches together with the fertility
rate in the other are large enough to sustain the population, then there exist
appropriate dispersal rates to compensate poor
local survival conditions.

The opposite result also holds. Adequately chosen dispersal rates, under
appropriate conditions, can make two isolated viable populations go extinct
when they are connected.

\begin{proposition}
\label{pr4} Let us assume $s_{i}^{1}=s_{i}^{2}=s_{i}$ for $i=2,3$, $R_{0}%
^{1}>1$ and $R_{0}^{2}>1$. If $\min\{s_{1}^{1}\phi^{2},s_{1}^{2}\phi
^{1}\}<(1-s_{2}s_{3})/s_{2}$ then there exist intervals $I_{1},I_{2}%
\subseteq\lbrack0,1]$ (resp. $\tilde{I}_{1},\tilde{I}_{2}\subseteq\lbrack
0,1]$) such that $\bar{R}_{0}<1$ (resp. $\tilde{R}_{0}<1$) for $v_{1}^{1}\in
I_{1}$ and $v_{2}^{1}\in I_{2}$ (resp. $v_{1}^{1}\in\tilde{I}_{1}$ and
$v_{2}^{1}\in\tilde{I}_{2}$).
\end{proposition}

\begin{proof}
Analogous to the proof of Proposition \ref{pr3}.
\end{proof}

\vspace{3ex}

\noindent\textbf{Effect of dispersal on reproductive synchrony.} We now
illustrate the influence of fast dispersal on the population's reproductive
synchrony. We do so in the particular case of homogenous patches, i.e., when
all demographic parameters are the same in both patches:
\begin{equation}
s_{i}^{1}=s_{i}^{2}=:s_{i}\text{ for }i=1,2,3;\ \phi^{1}=\phi^{2}%
=:\phi;\ c^{1}=c^{2}=:c;\ d^{1}=d^{2}=:d. \label{pr5-1}%
\end{equation}
Thus, the demographic local matrices coincide
\begin{equation}
\mathbf{D}^{1}(Y)=\mathbf{D}^{2}(Y)=\left(
\begin{array}
[c]{ccc}%
0 & s_{2}\phi\dfrac{1}{1+cy_{2}} & 0\\
s_{1} & 0 & s_{3}\dfrac{1}{1+dy_{2}}\\
0 & s_{2} & s_{3}\left(  1-\dfrac{1}{1+dy_{2}}\right) \\
&  &
\end{array}
\right)  , \label{eq111}%
\end{equation}
but, due to the dispersal rates, are different from the demographic global
matrices for systems \eqref{mod41a} and \eqref{mod42a}
\begin{equation}
\rule{-3ex}{0ex}\bar{H}(Y)=\left(
\begin{array}
[c]{ccc}%
0 & s_{2}\phi\left(  \dfrac{v_{2}^{1}}{1+cv_{2}^{1}y_{2}}+\dfrac{v_{2}^{2}%
}{1+cv_{2}^{2}y_{2}}\right)  & 0\\
\rule{0ex}{4ex}s_{1} & 0 & s_{3}\left(  \dfrac{v_{3}^{1}}{1+dv_{2}^{1}y_{2}%
}+\dfrac{v_{3}^{2}}{1+dv_{2}^{2}y_{2}}\right) \\
\rule{0ex}{4ex}0 & s_{2} & s_{3}\left(  1-\left(  \dfrac{v_{3}^{1}}%
{1+dv_{2}^{1}y_{2}}+\dfrac{v_{3}^{2}}{1+dv_{2}^{2}y_{2}}\right)  \right) \\
&  &
\end{array}
\right)  \label{eq112}%
\end{equation}
and
\[
\rule{-6ex}{0ex}\tilde{H}(Y)=\left(
\begin{array}
[c]{ccc}%
0 & s_{2}\phi\left(  \dfrac{v_{2}^{1}}{1+cs_{2}v_{2}^{1}y_{2}}+\dfrac
{v_{2}^{2}}{1+cs_{2}v_{2}^{2}y_{2}}\right)  & 0\\
\rule{0ex}{4ex}s_{1} & 0 & s_{3}\left(  \dfrac{v_{3}^{1}}{1+ds_{2}v_{2}%
^{1}y_{2}}+\dfrac{v_{3}^{2}}{1+ds_{2}v_{2}^{2}y_{2}}\right) \\
\rule{0ex}{4ex}0 & s_{2} & s_{3}\left(  1-\left(  \dfrac{v_{3}^{1}}%
{1+ds_{2}v_{2}^{1}y_{2}}+\dfrac{v_{3}^{2}}{1+ds_{2}v_{2}^{2}y_{2}}\right)
\right) \\
&  &
\end{array}
\right)
\]

As we pointed out in (\ref{inh21}), these four matrices have the same inherent
NRNs, $R_{0}^{1}=R_{0}^{2}=\bar{R}_{0}=\tilde{R}_{0}=\phi s_{1}s_{2}%
/(1-s_{2}s_{3})$, that we assume to be larger than 1. Then, the local
reproductive synchrony at both patches is determined by the sign of
\eqref{pr3-2}
\begin{equation}
a_{-}=s_{1}s_{2}s_{3}(1-s_{3})d-(1-s_{2}s_{3})s_{1}c, \label{pr5-2}%
\end{equation}
the global reproductive synchrony in the case of system \eqref{mod41} by the
sign of $\bar{a}_{-}$ in \eqref{pr2-2}
\begin{equation}
\bar{a}_{-}=s_{1}s_{2}s_{3}(1-s_{3})d\left(  v_{2}^{1}v_{3}^{1}+(1-v_{2}%
^{1})(1-v_{3}^{1})\right)  -(1-s_{2}s_{3})s_{1}c\left(  (v_{2}^{1}%
)^{2}+(1-v_{2}^{1})^{2}\right)  , \label{pr5-3}%
\end{equation}
and in the case of system \eqref{mod42} by the sign of $\tilde{a}_{-}$
\eqref{pr2-3}. Note that $\tilde{a}_{-}=s_{2}\bar{a}$ and so the sign of
$\tilde{a}_{-}$ coincides with that of $\bar{a}_{-}$, so that we will
concentrate on the sign of $\bar{a}_{-}$.

Therefore, in order to illustrate the influence of fast dispersal on
reproductive synchrony for both systems \eqref{mod41a} and \eqref{mod42a}, we
have to find conditions on $v_{2}^{1}$ and $v_{3}^{1}$ (note that $v_{1}^{1}$
plays no role whatsoever) such that $a_{-}$ and $\bar{a}_{-}$ have different
signs. In that way we show that adequately chosen adult dispersal rates can
alter the tendency to reproductive synchrony or asynchrony of isolated populations.

A first result shows that it is always possible to find adult dispersal rates
such that there exists global reproductive asynchrony, $\bar{a}_{-}<0$,
independently of the local tendency.

\begin{proposition}
\label{pr5} Let the model coefficients verify conditions \eqref{pr5-1}. Then
there exist intervals $I_{2},I_{3}\subseteq\lbrack0,1]$ such that $\bar{a}_{-}<0$
for $(v_{2}^{1}\in I_{2}$ and $v_{3}^{1})\in I_{3}$.
\end{proposition}

\begin{proof}
The expression \eqref{pr5-3} of $\bar{a}_{-}$ taken as a function of
$v_{2}^{1}$ and $v_{3}^{1}$, $\bar{a}_{-}(v_{2}^{1},v_{3}^{1})$, is continuous
and satisfies $\bar{a}_{-}(1,0)=-(1-s_{2}s_{3})s_{1}c<0$, what implies the
existence of the intervals $I_{2}$ and $I_{3}$ meeting the required conditions.
\end{proof}

This result proves that the local reproductive synchrony, $a_{-}>0$, can
always be changed into global reproductive asynchrony, $\bar{a}_{-}<0$,
through the appropriate adult dispersal rates. This fact is illustrated in
Figure \ref{fig2}.

\begin{figure}[h]
\centering
{\setlength{\fboxsep}{0pt}\setlength{\fboxrule}{0pt}\fbox{\includegraphics[width=0.9\textwidth]{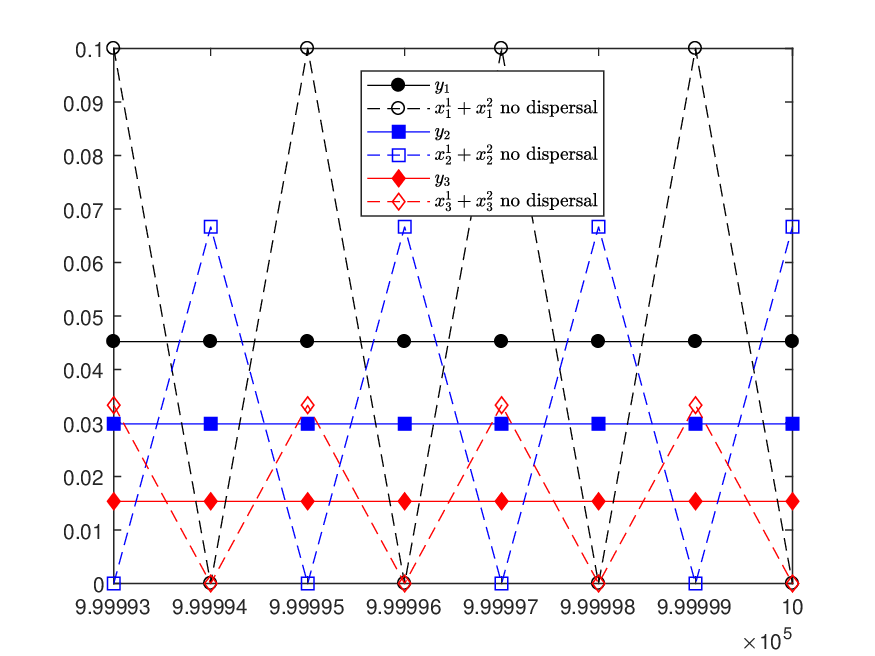}}%
} \caption{Changing local reproductive synchrony into global reproductive asynchrony.
Change from local reproductive synchrony to global reproductive
asynchrony in the adults reproductive behavior when going from the local level
(isolated identical patches that have dynamics defined by matrices
(\ref{eq111})) to the global level (represented by system (\ref{eq112})) when
we incorporate dispersal. \newline Top: parameter values $s_{1}=s_{2}%
=s_{3}=0.5$, $c=1$, $d=10$, $\phi=3.1$, $v_{1}^{1}=0.3$, $v_{2}^{1}%
=7/8,\ v_{3}^{1}=1/8$. \newline Initial conditions are
$X(0)=(0.02,0.02,0.05,0.05,0.02,0.02)$. The simulations have been run until
time $t=10^{6}$ and only the last 8 times are shown. }%
\label{fig2}
\end{figure}

In the opposite sense we present the next result in which we provide
sufficient conditions so that the local reproductive asynchrony, $a_{-}<0$,
can be changed into global reproductive synchrony, $\bar{a}_{-}>0$, if adult
dispersal is adequately chosen.

\begin{proposition}
\label{pr6} Let the model coefficients verify conditions \eqref{pr5-1} and
$a_{-}<0$. Then, there exist intervals $I_{2},I_{3}\subseteq\lbrack0,1]$ such
that $\bar{a}_{-}>0$ for $v_{2}^{1}\in I_{2}$, $,v_{3}^{1}\in I_{3}$ if and
only if
\[
(1-s_{2}s_{3})s_{1}c<\frac{1+\sqrt{2}}{2}s_{1}s_{2}s_{3}(1-s_{3})d.
\]

\end{proposition}

\begin{proof}
The expression \eqref{pr5-3} of $\bar{a}_{-}$ taken as a function of
$v_{2}^{1}$ and $v_{3}^{1}$, $\bar{a}_{-}(v_{2}^{1},v_{3}^{1})$, is continuous
and satisfies
\[%
\begin{array}
[c]{rcl}%
\bar{a}_{-}(1-\dfrac{\sqrt{2}}{2},0) & = & s_{1}s_{2}s_{3}(1-s_{3}%
)d\dfrac{\sqrt{2}}{2}-(1-s_{2}s_{3})s_{1}c(4-2\sqrt{2})\\
& = & (4-2\sqrt{2})\left(  s_{1}s_{2}s_{3}(1-s_{3})d\dfrac{1+\sqrt{2}}%
{2}-(1-s_{2}s_{3})s_{1}c\right)  >0,
\end{array}
\]
what implies the existence of the intervals $I_{2}$ and $I_{3}$ meeting the
required conditions.

In the opposite sense, if we assume, by contradiction, that
\[
(1-s_{2}s_{3})s_{1}c\geq\dfrac{1+\sqrt{2}}{2}s_{1}s_{2}s_{3}(1-s_{3})d
\]
and $\bar{a}_{-}>0$ then
\[
\dfrac{1+\sqrt{2}}{2}\leq\frac{(1-s_{2}s_{3})s_{1}c}{s_{1}s_{2}s_{3}%
(1-s_{3})d}<\frac{v_{2}^{1}v_{3}^{1}+(1-v_{2}^{1})(1-v_{3}^{1})}{(v_{2}%
^{1})^{2}+(1-v_{2}^{1})^{2}}%
\]
but there are no $v_{2}^{1},v_{3}^{1}\in\lbrack0,1]$ satisfying the previous
inequality since the maximum of function $f(x,y)=(xy+(1-x)(1-y))/(x^{2}%
+(1-x)^{2})$ on $[0,1]\times\lbrack0,1]$ is $(1+\sqrt{2})/2$.
\end{proof}

This change from local reproductive asynchrony to global reproductive synchrony
in the adults reproductive behavior for adequate dispersal parameters is
illustrated in Figure (\ref{fig3})

\begin{figure}[h]
\centering
{\setlength{\fboxsep}{0pt}\setlength{\fboxrule}{0pt}\fbox{\includegraphics[width=0.9\textwidth]{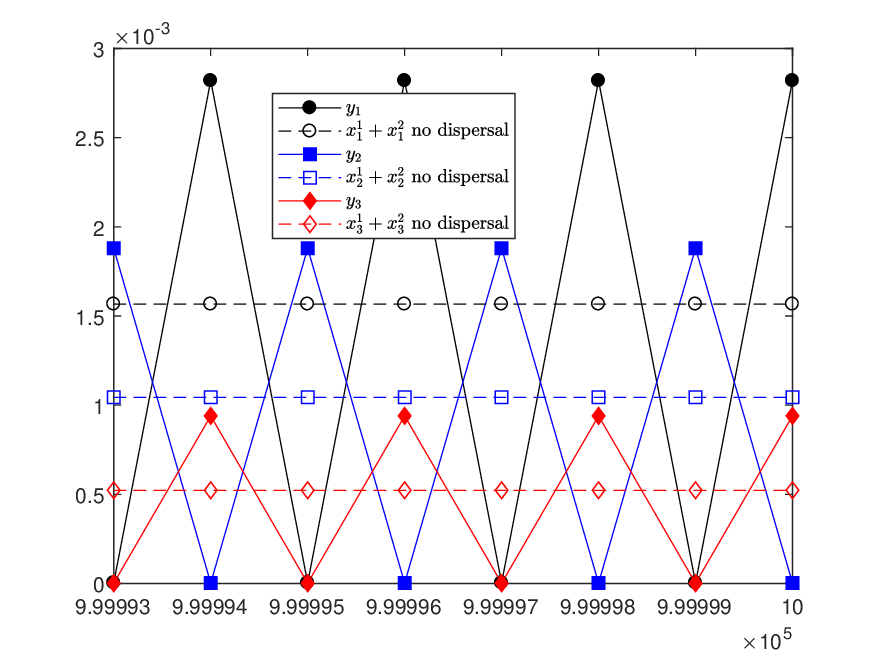}}%
} \caption{Changing local reproductive asynchrony into global reproductive synchrony. Change from local reproductive asynchrony to global reproductive
synchrony in the adults reproductive behavior when going from the local level
(isolated identical patches that have dynamics defined by matrices
(\ref{eq111})) to the global level (represented by system (\ref{eq112})) when
we incorporate dispersal. \newline Parameter values: $s_{1}=s_{2}=s_{3}=0.5$,
$c=1$, $d=5.5$, $\phi=3.0003$, $v_{1}^{1}=0.3,\ v_{2}^{1}=3/8,\ v_{3}^{1}%
=1/8$. \newline Initial conditions are $X(0)=(0.02,0.02,0.05,0.05,0.02,0.02)$.
The simulations have been run until time $t=10^{6}$ and only the last 8 times
are shown. }
\label{fig3}
\end{figure}

The preceding behavior, in which if appropriate adult dispersal rates are
chosen, the existence of a positive locally asymptotically stable equilibrium
locally in each site can be transformed globally into the existence of a
stable 2-cycle and viceversa, is analogous to the one found in
\cite{bravo2016fast} for a semelparous population structured in juveniles and
adults, spread out in two patches between which they can migrate.

\subsection{Differences between models \eqref{mod41} and \eqref{mod42}.}

\label{sec42} In this section we propose some particular simple settings in
which the asymptotic outcomes of systems \eqref{mod41} and \eqref{mod42}
differ, thus showing that the decision of placing mortality at the slow or the
fast time scale can have crucial consequences.

There is a simple case in which the inherent NRNs $\bar{R}_{0}$ and $\tilde
{R}_{0}$, associated to systems \eqref{mod41} and \eqref{mod42}, coincide.
Indeed, if $s_{i}^{1}=s_{i}^{2}=:s_{i}$, for $i=1,2,3$, then
\[
\tilde{R}_{0}=\frac{s_{1}s_{2}\left(  \phi^{1}v_{2}^{1}+\phi^{2}v_{2}%
^{2}\right)  }{1-s_{2}s_{3}}=\bar{R}_{0}.
\]

In the next result we show that equal fertilities, $\phi^{1}=\phi^{2}$, imply
that $\bar{R}_{0}$ is larger than $\tilde{R}_{0}$.

\begin{proposition}
\label{pr7} In systems \eqref{mod41} and \eqref{mod42}, if $\phi^{1}=\phi
^{2}=:\phi$ then
\[
\tilde{R}_{0}<\bar{R}_{0}.
\]
Moreover, there is an interval $I\in\mathbb{R}_{+}$ such that if $\phi\in I $
then
\begin{equation}
\tilde{R}_{0}<1<\bar{R}_{0}. \label{ineq100}%
\end{equation}

\end{proposition}

\begin{proof}
If we use the fact that, for $i=1,2,3$, we have $v_{i}^{1},v_{i}^{2}\in(0,1) $
and $v_{i}^{1}+v_{i}^{2}=1$, the inequality relating the weighted arithmetic
mean and the weighted geometric mean implies that the survival rates in
\eqref{mod41a} and \eqref{mod42a} verify
\[
\tilde{s}_{i}=\left(  s_{i}^{1}\right)  ^{v_{i}^{1}}\left(  s_{i}^{2}\right)
^{v_{i}^{2}}<s_{i}^{1}v_{i}^{1}+s_{i}^{2}v_{i}^{2}=\bar{s}_{i},
\]
and the fertility coefficients
\[
\tilde{b}=\tilde{s}_{2}\left(  \phi v_{2}^{1}+\phi v_{2}^{2}\right)
=\tilde{s}_{2}\phi<\bar{s}_{2}\phi=\phi s_{2}^{1}v_{2}^{1}+\phi s_{2}^{2}%
v_{2}^{2}=\bar{b}.
\]
Now, it is straightforward that
\[
\tilde{R}_{0}=\frac{\tilde{b}\tilde{s}_{1}}{1-\tilde{s}_{2}\tilde{s}_{3}%
}<\frac{\bar{b}\bar{s}_{1}}{1-\bar{s}_{2}\bar{s}_{3}}=\bar{R}_{0}%
\]
Inequality (\ref{ineq100}) is a direct consequence of both $\tilde{R}_{0}$ and
$\bar{R}_{0}$ depending linearly on $\phi$.
\end{proof}

We stress that choosing one of models \eqref{mod41} and \eqref{mod42} over the
other can represent the difference between population survival or extinction.

We now present a situation in which the inequalities in Proposition \ref{pr7}
are reversed. We assume that $s_{1}^{1}=s_{1}^{2}=s_{1}$, $s_{1}^{2}=0.7$,
$s_{2}^{2}=0.5$, $s_{3}^{1}=s_{3}^{2}=0.8$, $v_{2}^{1}=v_{2}^{2}=0.5$,
$\phi^{1}=\phi$ and $\phi^{2}=\alpha\phi$. Straightforward calculations yield
\[
\tilde{R}_{0}-\bar{R}_{0}=s_{1}\phi\left(  \left(  \frac{5\sqrt{35}%
}{100-8\sqrt{35}}-\frac{25}{52}\right)  \alpha+\frac{5\sqrt{35}}%
{100-8\sqrt{35}}-\frac{35}{52}\right)  ,
\]
therefore, for $\alpha>\dfrac{65\sqrt{35}+14}{289}$ we have $\bar{R}%
_{0}<\tilde{R}_{0}$. As this hold for any values of $s_{1}\in(0,1)$ and
$\phi>0$, there are some of them for which
\[
\bar{R}_{0}<1<\tilde{R}_{0},
\]
which represents the difference between population survival or extinction but
exchanging the associated models in Proposition \ref{pr7}.

It can be shown that, similarly to what happens regarding survival, there
exist situations where models \eqref{mod41} and \eqref{mod42} have different
outputs regarding the population reproductive synchrony. Nevertheless, we skip
this analysis since we consider that the previous results illustrate clearly
our main point here: the choice of model should be as accurate as possible
because the results can drastically differ.

\section{Discussion}

\label{sec5}

In this work we have extended the results on reduction of two time scales non
linear discrete systems presented in \cite{Sanz08}. The new result avoids the
need to check a difficult hypothesis, specifically the uniform convergence on
compacts sets of the differentials of the iterates of a map. The dropping this
hypothesis has an effect on the results that can be given regarding the
relationships between the original and the reduced model, but nonetheless can
be viewed as a minor effect with regard to the practical study of population
dynamics models.

This new result has opened the door to extending to the nonlinear case the
work on re-scaling developed in \cite{Nguyen11}. When there are two processes
acting at different time scales that must be gathered in a single discrete
model, the easiest choice of its time unit is that associated to the slow one.
The structure of the model then reflects that the fast process acts a number
of times, approximately equal to the ratio between the two time scales,
followed by one action of the slow process. It is not always easy to decide if
a process occurs at the slow or the fast time scale. Here we have focussed on
the re-scaling of survival, which is a process usually measured at the slow
time scale associated to the rest of the demographic processes. Nevertheless,
in a context of fast movements of individuals between patches with different
associated survival rates, it can be argued that it should be rather
considered as occurring at the fast time scale. Systems (\ref{mod41}) and
(\ref{mod42}) represent general discrete time models of structured
metapopulations with two time scales; in the first one survival acts at the
slow time scale and in the second one the re-scaled survival acts at the fast
time scale. The reduction results developed in Sections (\ref{sec2}) and
(\ref{sec3}) associated to these two models yield two aggregated models,
(\ref{mod41a}) and (\ref{mod42a}), that contain the asymptotic information of
the original models. They summarize the global emergent properties
\cite{Auger08a,Auger98} that fast dispersal induce out of local demography.

To illustrate this influence of fast dispersal on local demographic dynamics
and the relevance of the choice of time scale for the different processes
involved in a model, we have proposed a particular case of systems
(\ref{mod41}) and (\ref{mod42}). It is based in the model in
\cite{Veprauskas17} and it has three population stages and two patches. The
comparison of the local models and the two global models is done through their
respective inherent net reproduction numbers, that decide on the survival or
extinction of the population, and the coefficients $a_{-}$ that rule, for
viable populations, the tendency to either reproductive synchrony or
asynchrony. Even in simple cases it can be shown that viable local populations
can get globally extinct for adequate dispersal rates and, the other way
round, non viable local populations can globally survive if migrating
appropriately. Along the same lines, different scenarios are presented
reversing the outcome of reproductive synchrony/asynchrony between local and
global dynamics.

Finally, it is shown that there are some cases where changing from system
(\ref{mod41}) to (\ref{mod42}) (or viceversa) can represent changing the
global outcome from survival to extinction or the other way round. This fact
stresses the importance of the choice of time scale (slow ar fast) in which
survival is included in the model.

It must be stressed that the result in Theorem 2 can be applied to much more
general systems than the specific structured metapopulation models presented
in Section 3. Also, the convergence results in Proposition \ref{prop:conv}
have a more general application than the survival re-scaling proposed in this
work. We have chosen the context of the work trying to keep at reasonable
levels both simplicity and modelling relevance.

An interesting extension of the treated metapopulation models would also
include an epidemic disease dynamics together with the demographic and spatial issues.
These more general models would encompass time scales too. A relevant issue at
this point would be to decide whether the epidemic process must be considered
as acting as at the fast or the slow time scale. In case that it be considered
as a slow process together with demography, the reduction of the corresponding
two time scales system would not differ much from what has been developed in
this work since it mainly depends on the fast process. On the other hand, a
general assumption in basic epidemic models is that disease evolution can be
considered almost instantaneous with respect to demography and therefore this
latter is considered negligible. A different approach to this last assumption
is considering in the same model both disease and demographic dynamics at
different time scales. The obtained two-time scale systems would be
susceptible of reduction. The inclusion of disease dynamics in the fast part
of the system would render the reduction procedure more involved, and Theorem
\ref{th1} should be a tool helping in this task.

\appendix

\section{Appendix}

\label{appx}

\begin{lemma}
\label{lema1} Let Hypotheses \ref{H1} and \ref{H3} hold. Then for any
$m\in\mathbb{N}$ we have $\lim_{k\rightarrow\infty}H_{k}^{m}=H^{m}$ uniformly
on compact sets of $\Omega_{N}$.
\end{lemma}

\begin{proof}
It is easy to realize that it suffices with proving the result for $m=2$.

Let $M\subset\Omega_{n}$ be a compact set. Since $\lim_{k\rightarrow\infty
}H_{k}=H$ uniformly on $M$, we can find another compact set $C\subset
\Omega_{n}$ and $k^{*}\in\mathbb{N}$ such that $H(M)\subset C$ and
$H_{k}(M)\subset C$ for all $k\geq k^{*}$.

Let us now show that $\lim_{k\rightarrow\infty}H_{k}^{2}=H^{2}$ uniformly on
$K$.

The uniform convergence $\lim_{k\rightarrow\infty}H_{k}=H$ on $C$ assures the
existence of a real sequence $\left\{  \alpha_{k}\right\}  _{k\in\mathbb{N}}$,
$\alpha_{k}>0$, with $\lim_{k\rightarrow\infty}\alpha_{k}=0$ and such that
$\sup_{X\in C}\left\Vert H_{k}(X)-H(X)\right\Vert \leq\alpha_{k}$. Since
\[
\left\Vert H_{k}^{2}(X)-H^{2}(X)\right\Vert \leq\left\Vert H_{k}%
(H_{k}(X))-H(H_{k}(X\mathbf{)})\right\Vert +\left\Vert H(H_{k}%
(X))-H(H(X\mathbf{)})\right\Vert ,
\]
we have, for $k\geq k^{\ast}$, that
\[
\underset{X\in M}{\sup}\left\Vert H_{k}^{2}(X)-H^{2}(X)\right\Vert
\leq\underset{Z\in C}{\sup}\left\Vert H_{k}(Z)-H(Z)\right\Vert +\sup_{{\tiny
\begin{array}
[c]{c}%
Z_{k},Z\in C\\
\left\Vert Z_{k}-Z\right\Vert \leq\alpha_{k}%
\end{array}
}}\left\Vert H(Z_{k})-H(Z)\right\Vert .
\]
When $k\rightarrow\infty$, the first term on the right-hand side converges to
zero due to the uniform convergence of $H_{k}$ to $H$ on $C$ and the second
term converges to zero since $H$ is uniformly continuous on $C$. Therefore the
result is proved.
\end{proof}

\begin{proposition}
\label{prop:conv} Let $I,J\in\mathbb{Z}_{+}$ and let $\Omega\subset
\mathbb{R}^{I}$ be an open set. Let $M:\Omega\rightarrow\mathbb{R}^{J\times
J}$ and $S:\Omega\rightarrow\mathbb{R}^{J\times J}$ be continuous maps such that:

a. For all $Y\in\Omega$, $M(Y)$ is a primitive probability matrix. Let $v(Y)$
be its column Perron right eigenvector normalized so that $\mathbf{1}v(Y)=1$.

b. There exists a continuous map $S^{\prime}:\Omega\rightarrow\mathbb{R}%
^{J\times J}$ such that $\exp\left(  S^{\prime}(Y)\right)  =S(Y)$ for all
$Y\in\Omega$.

Let us define, for $Y\in\Omega$, $S_{k}(Y):=\exp\left(  \frac{1}{k}S^{\prime
}(Y)\right)  $, $\bar{M}(Y):=v(Y)\mathbf{1}$ and $\gamma(Y):=\exp\left(
\mathbf{1}S^{\prime}(Y)v(Y)\right)  $. Then we have
\begin{equation}
\lim_{k\rightarrow\infty}\left(  S_{k}(Y)M(Y)\right)  ^{k}=\lim_{k\rightarrow
\infty}\left(  M(Y)S_{k}(Y)\right)  ^{k}=\gamma(Y)\bar{M}(Y)
\label{limite-uniforme}%
\end{equation}
where the limit is uniform on any compact set of $\Omega$.
\end{proposition}

\begin{proof}
Given a fixed $Y\in\Omega$, matrix $M(Y)$ is primitive, its (strictly)
dominant eigenvalue is 1 and its associated left eigenvector $u(Y),$
normalized so that $u(Y)^{T}v(Y)=1,$ is $u(Y)=\mathbf{1}$. Then, the existence
of the pointwise limit (\ref{limite-uniforme}) follows as a particular case of
Theorem A1 in \cite{Nguyen11}. Therefore, all we need to prove is that limit
(\ref{limite-uniforme}) is uniform on compact sets of $\Omega$.

We only need to adjust some details at the beginning of the proof of Theorem
A1 in \cite{Nguyen11} to make it work in the present case via standard
compactness arguments. Let $|\rule{-0.4ex}{0ex}|\rule{-0.4ex}{0ex}%
|\ |\rule{-0.4ex}{0ex}|\rule{-0.4ex}{0ex}|$ be the 1-norm in $\mathbb{R}%
^{J\times J}$. In the first place, since the $M(Y)$ are probability matrices
one has
\begin{equation}
|\rule{-0.4ex}{0ex}|\rule{-0.4ex}{0ex}|M(Y)|\rule{-0.4ex}{0ex}%
|\rule{-0.4ex}{0ex}|=1 \label{dem03}%
\end{equation}
for all $Y\in\Omega$. Now let $K\subset\Omega$ be any compact set. For any
$Y^{\prime}\in K$ there exists $r_{Y^{\prime}}<1$ such that
$|\rule{-0.4ex}{0ex}|\rule{-0.4ex}{0ex}|M(Y^{\prime})-\bar{M}(Y^{\prime
})|\rule{-0.4ex}{0ex}|\rule{-0.4ex}{0ex}|<r_{Y^{\prime}}$. The continuity of
$M$ and of the norm imply that if $\bar{r}_{Y^{\prime}}$ satisfies
$r_{Y^{\prime}}<\bar{r}_{Y^{\prime}}<1$, there exists an open neighbourhood of
$Y^{\prime}$ in $\Omega$, $\mathcal{U}_{Y^{\prime}}$, such that
\begin{equation}
|\rule{-0.4ex}{0ex}|\rule{-0.4ex}{0ex}|M(Y)-\bar{M}(Y)|\rule{-0.4ex}{0ex}%
|\rule{-0.4ex}{0ex}|<\bar{r}_{Y^{\prime}}\text{ for all }Y\in\mathcal{U}%
_{Y^{\prime}}. \label{dem01}%
\end{equation}

Let $\delta>0$ and let us denote%
\begin{align*}
\left(  \delta S(Y)\right)  ^{\prime}  &  :=S^{\prime}(Y)+\log\delta I\\
\left(  \delta S(Y)\right)  _{k}  &  :=\exp(\frac{1}{k}\left(  \delta
S(Y)\right)  ^{\prime}),\ \gamma_{\delta S}(Y):=\exp\left(  \mathbf{1}\left(
\delta S(Y)\right)  ^{\prime}v(Y)\right)
\end{align*}
It is immediate to check that $\left(  \delta S(Y)\right)  _{k}=\delta
^{\frac{1}{k}}S_{k}(Y)$ and $\gamma_{\delta S}(Y)=\delta\gamma(Y).$ Therefore%
\[
\left(  \left(  \delta S(Y)\right)  _{k}M(Y)\right)  ^{k}-\gamma_{\delta
S}(Y)=\delta\left(  S_{k}(Y)M(Y)-\gamma(Y)\right)  ^{k}%
\]
so that the limit (\ref{limite-uniforme}) is uniform in $K$ if and only if
there exists $\delta>0$ such that the limit $\lim_{k\rightarrow\infty}\left(
\left(  \delta S(Y)\right)  _{k}M(Y)\right)  ^{k}=\gamma_{\delta S}(Y)$ is
uniform in $K$. Now let $\delta:=\min_{Y\in\mathcal{\bar{U}}_{Y^{\prime}}}%
\exp\left(  -|\rule{-0.4ex}{0ex}|\rule{-0.4ex}{0ex}|S^{\prime}%
(Y)|\rule{-0.4ex}{0ex}|\rule{-0.4ex}{0ex}|\right)  >0$. Then,
\[
|\rule{-0.4ex}{0ex}|\rule{-0.4ex}{0ex}|\left(  \delta S(Y)\right)
_{k}|\rule{-0.4ex}{0ex}|\rule{-0.4ex}{0ex}|=\delta^{\frac{1}{k}}%
|\rule{-0.4ex}{0ex}|\rule{-0.4ex}{0ex}|\exp\left(  \frac{1}{k}S^{\prime
}(Y)\right)  |\rule{-0.4ex}{0ex}|\rule{-0.4ex}{0ex}|\leq1
\]
and so we can assume, without loss of generality, that $|\rule{-0.4ex}{0ex}%
|\rule{-0.4ex}{0ex}|S_{k}(Y)|\rule{-0.4ex}{0ex}|\rule{-0.4ex}{0ex}|\leq1$ for
all $Y\in\mathcal{U}_{Y^{\prime}}$ and all $k=1,2,...$, from where it follows
that
\begin{equation}
|\rule{-0.4ex}{0ex}|\rule{-0.4ex}{0ex}|\left(  M(Y)S_{k}(Y)\right)
^{i}|\rule{-0.4ex}{0ex}|\rule{-0.4ex}{0ex}|\leq|\rule{-0.4ex}{0ex}%
|\rule{-0.4ex}{0ex}|M(Y)|\rule{-0.4ex}{0ex}|\rule{-0.4ex}{0ex}|^{i}%
|\rule{-0.4ex}{0ex}|\rule{-0.4ex}{0ex}|S_{k}(Y)|\rule{-0.4ex}{0ex}%
|\rule{-0.4ex}{0ex}|^{i}=|\rule{-0.4ex}{0ex}|\rule{-0.4ex}{0ex}|S_{k}%
(Y)|\rule{-0.4ex}{0ex}|\rule{-0.4ex}{0ex}|^{i}\leq1 \label{dem02}%
\end{equation}
for all $Y\in\mathcal{U}_{Y^{\prime}}$ and $i,k\in\mathbb{N}.$ Thus, from
(\ref{dem03}), (\ref{dem01}) and (\ref{dem02}) the rest of the proof of
Theorem A1 in \cite{Nguyen11} is valid (in the particular case of working with
the 1-matrix norm) and yields that limit (\ref{limite-uniforme}) is uniform on
$\mathcal{U}_{Y^{\prime}}$. A standard compactness argument ensures the
uniform convergence on $K$ and completes the proof.

\end{proof}

\section*{Competing interests}
The authors declare that they have no competing interests.

\section*{Author's contributions}
    All authors have the same contributions. All authors read and approved the final manuscript.

\section*{Funding}
This work was supported by Ministerio de Economía y Competitividad (Spain), project MTM2014-56022-C2-1-P.
\bibliographystyle{bmc-mathphys}
\bibliography{ADE2019}

\begin{thebibliography}{20}
\expandafter\ifx\csname natexlab\endcsname\relax\def\natexlab#1{#1}\fi
\providecommand{\url}[1]{\texttt{#1}}
\providecommand{\href}[2]{#2}
\providecommand{\path}[1]{#1}
\providecommand{\DOIprefix}{doi:}
\providecommand{\ArXivprefix}{arXiv:}
\providecommand{\URLprefix}{URL: }
\providecommand{\Pubmedprefix}{pmid:}
\providecommand{\doi}[1]{\href{http://dx.doi.org/#1}{\path{#1}}}
\providecommand{\Pubmed}[1]{\href{pmid:#1}{\path{#1}}}
\providecommand{\bibinfo}[2]{#2}
\ifx\xfnm\relax \def\xfnm[#1]{\unskip,\space#1}\fi
\bibitem[{Auger and Poggiale(1998)}]{Auger98}
\bibinfo{author}{Auger, A.}, \bibinfo{author}{Poggiale, J.},
  \bibinfo{year}{1998}.
\newblock \bibinfo{title}{Aggregation and emergence in systems of ordinary
  differential equations}.
\newblock \bibinfo{journal}{Mathematical and Computer Modelling}
  \bibinfo{volume}{27}, \bibinfo{pages}{1--22}.
\bibitem[{Auger and Lett(2003)}]{Auger03}
\bibinfo{author}{Auger, P.}, \bibinfo{author}{Lett, C.}, \bibinfo{year}{2003}.
\newblock \bibinfo{title}{Integrative biology: linking levels of organization}.
\newblock \bibinfo{journal}{Comptes Rendus de l’Académie des Sciences de
  Paris, Biology} \bibinfo{volume}{326}, \bibinfo{pages}{517--522}.
\bibitem[{Auger et~al.(2008a)Auger, Bravo de~la Parra, Poggiale, S\'anchez and
  Nguyen-Huu}]{Auger08a}
\bibinfo{author}{Auger, P.}, \bibinfo{author}{Bravo de~la Parra, R.},
  \bibinfo{author}{Poggiale, J.C.}, \bibinfo{author}{S\'anchez, E.},
  \bibinfo{author}{Nguyen-Huu, T.}, \bibinfo{year}{2008}a.
\newblock \bibinfo{title}{Aggregation of variables and applications to
  population dynamics}, in: \bibinfo{booktitle}{Structured population models in
  biology and epidemiology}. \bibinfo{publisher}{Springer},
  \bibinfo{address}{Berlin}, pp. \bibinfo{pages}{209--263}.
\bibitem[{Auger et~al.(2008b)Auger, Bravo de~la Parra, Poggiale, S\'anchez and
  Sanz}]{Auger08b}
\bibinfo{author}{Auger, P.}, \bibinfo{author}{Bravo de~la Parra, R.},
  \bibinfo{author}{Poggiale, J.C.}, \bibinfo{author}{S\'anchez, E.},
  \bibinfo{author}{Sanz, L.}, \bibinfo{year}{2008}b.
\newblock \bibinfo{title}{Aggregation methods in dynamical systems and
  applications in population and community dynamics}.
\newblock \bibinfo{journal}{Physics of Life Reviews} \bibinfo{volume}{5},
  \bibinfo{pages}{79--105}.
\bibitem[{Border(1989)}]{border1989fixed}
\bibinfo{author}{Border, K.C.}, \bibinfo{year}{1989}.
\newblock \bibinfo{title}{Fixed point theorems with applications to economics
  and game theory}.
\newblock \bibinfo{publisher}{Cambridge university press},
  \bibinfo{address}{Cambridge}.
\bibitem[{Cushing(2015)}]{Cushing15}
\bibinfo{author}{Cushing, J.}, \bibinfo{year}{2015}.
\newblock \bibinfo{title}{Mathematics of Planet Earth: Dynamics, Games and
  Science}. \bibinfo{publisher}{Springer}, \bibinfo{address}{Berlin}. chapter
  \bibinfo{chapter}{On the fundamental bifurcation theorem for semelparous
  Leslie models}.
\newblock CIM Mathematical Sciences Series, pp. \bibinfo{pages}{215--251}.
\bibitem[{Cushing and Yicang(1994)}]{cushing1994net}
\bibinfo{author}{Cushing, J.}, \bibinfo{author}{Yicang, Z.},
  \bibinfo{year}{1994}.
\newblock \bibinfo{title}{The net reproductive value and stability in matrix
  population models}.
\newblock \bibinfo{journal}{Natural Resources Modeling} \bibinfo{volume}{8},
  \bibinfo{pages}{297--333}.
\bibitem[{Franco and Ruiz-Herrera(2015)}]{franco2015connect}
\bibinfo{author}{Franco, D.}, \bibinfo{author}{Ruiz-Herrera, A.},
  \bibinfo{year}{2015}.
\newblock \bibinfo{title}{To connect or not to connect isolated patches}.
\newblock \bibinfo{journal}{Journal of theoretical biology}
  \bibinfo{volume}{370}, \bibinfo{pages}{72--80}.
\bibitem[{Jang and Mitra(2000)}]{jang2000equilibrium}
\bibinfo{author}{Jang, S.J.}, \bibinfo{author}{Mitra, A.K.},
  \bibinfo{year}{2000}.
\newblock \bibinfo{title}{Equilibrium stability of single-species
  metapopulations}.
\newblock \bibinfo{journal}{Bulletin of mathematical biology}
  \bibinfo{volume}{62}, \bibinfo{pages}{155--161}.
\bibitem[{LaSalle(1976)}]{Lasalle76}
\bibinfo{author}{LaSalle, J.P.}, \bibinfo{year}{1976}.
\newblock \bibinfo{title}{The stability of dynamical systems}.
  volume~\bibinfo{volume}{25}.
\newblock \bibinfo{publisher}{Siam}, \bibinfo{address}{Philadelphia}.
\bibitem[{Levin(1992)}]{Levin92}
\bibinfo{author}{Levin, S.}, \bibinfo{year}{1992}.
\newblock \bibinfo{title}{The problem of pattern and scale in ecology}.
\newblock \bibinfo{journal}{Ecology} \bibinfo{volume}{73},
  \bibinfo{pages}{1943--1967}.
\bibitem[{Lischke et~al.(2007)}]{Lischke07}
\bibinfo{author}{Lischke, H.}, et~al., \bibinfo{year}{2007}.
\newblock \bibinfo{title}{Up-scaling of biological properties and models to the
  landscape level}.
\bibitem[{Marv\'a et~al.(2013)Marv\'a, Moussaoui, Bravo de~la Parra and
  Auger}]{Marva13}
\bibinfo{author}{Marv\'a, M.}, \bibinfo{author}{Moussaoui, A.},
  \bibinfo{author}{Bravo de~la Parra, R.}, \bibinfo{author}{Auger, P.},
  \bibinfo{year}{2013}.
\newblock \bibinfo{title}{A density-dependent model describing age-structured
  population dynamics using hawk--dove tactics}.
\newblock \bibinfo{journal}{Journal of Difference Equations and Applications}
  \bibinfo{volume}{19}, \bibinfo{pages}{1022--1034}.
\bibitem[{Marv\'a et~al.(2009)Marv\'a, S\'anchez, Bravo de~la Parra and
  Sanz}]{Marva09}
\bibinfo{author}{Marv\'a, M.}, \bibinfo{author}{S\'anchez, E.},
  \bibinfo{author}{Bravo de~la Parra, R.}, \bibinfo{author}{Sanz, L.},
  \bibinfo{year}{2009}.
\newblock \bibinfo{title}{Reduction of slow--fast discrete models coupling
  migration and demography}.
\newblock \bibinfo{journal}{Journal of theoretical biology}
  \bibinfo{volume}{258}, \bibinfo{pages}{371--379}.
\bibitem[{Nguyen-Huu et~al.(2011)Nguyen-Huu, Bravo de~la Parra and
  Auger}]{Nguyen11}
\bibinfo{author}{Nguyen-Huu, T.}, \bibinfo{author}{Bravo de~la Parra, R.},
  \bibinfo{author}{Auger, P.}, \bibinfo{year}{2011}.
\newblock \bibinfo{title}{Approximate aggregation of linear discrete models
  with two time scales: re-scaling slow processes to the fast scale}.
\newblock \bibinfo{journal}{Journal of Difference Equations and Applications}
  \bibinfo{volume}{17}, \bibinfo{pages}{621--635}.
\bibitem[{Bravo de~la Parra et~al.(2013)Bravo de~la Parra, Marv\'a, S\'anchez
  and Sanz}]{Bravo13}
\bibinfo{author}{Bravo de~la Parra, R.}, \bibinfo{author}{Marv\'a, M.},
  \bibinfo{author}{S\'anchez, E.}, \bibinfo{author}{Sanz, L.},
  \bibinfo{year}{2013}.
\newblock \bibinfo{title}{Reduction of discrete dynamical systems with
  applications to dynamics population models}.
\newblock \bibinfo{journal}{Mathematical Modelling of Natural Phenomena}
  \bibinfo{volume}{8}, \bibinfo{pages}{107--129}.
\bibitem[{Bravo de~la Parra et~al.(2016)Bravo de~la Parra, Marv{\'a} and
  Sansegundo}]{bravo2016fast}
\bibinfo{author}{Bravo de~la Parra, R.}, \bibinfo{author}{Marv{\'a}, M.},
  \bibinfo{author}{Sansegundo, F.}, \bibinfo{year}{2016}.
\newblock \bibinfo{title}{Fast dispersal in semelparous populations}.
\newblock \bibinfo{journal}{Mathematical Modelling of Natural Phenomena}
  \bibinfo{volume}{11}, \bibinfo{pages}{120--134}.
\bibitem[{Sanz et~al.(2008)Sanz, Bravo de~la Parra and S\'anchez}]{Sanz08}
\bibinfo{author}{Sanz, L.}, \bibinfo{author}{Bravo de~la Parra, R.},
  \bibinfo{author}{S\'anchez, E.}, \bibinfo{year}{2008}.
\newblock \bibinfo{title}{Approximate reduction of non-linear discrete models
  with two time scales}.
\newblock \bibinfo{journal}{Journal of Difference Equations and Applications}
  \bibinfo{volume}{14}, \bibinfo{pages}{607--627}.
\bibitem[{Veprauskas and Cushing(2017)}]{Veprauskas17}
\bibinfo{author}{Veprauskas, A.}, \bibinfo{author}{Cushing, J.M.},
  \bibinfo{year}{2017}.
\newblock \bibinfo{title}{A juvenile--adult population model: climate change,
  cannibalism, reproductive synchrony, and strong allee effects}.
\newblock \bibinfo{journal}{Journal of biological dynamics}
  \bibinfo{volume}{11}, \bibinfo{pages}{1--24}.
\bibitem[{Yakubu(2008)}]{yakubu2008asynchronous}
\bibinfo{author}{Yakubu, A.A.}, \bibinfo{year}{2008}.
\newblock \bibinfo{title}{Asynchronous and synchronous dispersals in spatially
  discrete population models}.
\newblock \bibinfo{journal}{SIAM Journal on Applied Dynamical Systems}
  \bibinfo{volume}{7}, \bibinfo{pages}{284--310}.

\end{thebibliography}

\end{document}